\documentclass[12pt, reqno]{amsart}
\usepackage{amsmath, amsthm, amscd, amsfonts, amssymb, graphicx, color, mathrsfs}
\usepackage[bookmarksnumbered, colorlinks, plainpages]{hyperref}
\usepackage[all]{xy}
\usepackage{slashed}

\newtheorem{theorem}{Theorem}[section]
\newtheorem{lemma}[theorem]{Lemma}
\newtheorem{proposition}[theorem]{Proposition}
\newtheorem{corollary}[theorem]{Corollary}
\theoremstyle{definition}

\theoremstyle{remark}
\newtheorem{remark}[theorem]{Remark}
\numberwithin{equation}{section}

\begin{document}
\setcounter{page}{1}

\title[ H\"ormander condition for  pseudo-multipliers  ]{ H\"ormander condition for  pseudo-multipliers associated to the harmonic oscillator}

\author[D. Cardona]{Duv\'an Cardona}
\address{
  Duv\'an Cardona:
  \endgraf
  Department of Mathematics
  \endgraf
  Pontificia Universidad Javeriana
  \endgraf
  Bogot\'a
  \endgraf
  Colombia
  \endgraf
  {\it E-mail address} {\rm duvanc306@gmail.com;
cardonaduvan@javeriana.edu.co}
  }

\author[M. Ruzhansky]{Michael Ruzhansky}
\address{
  Michael Ruzhansky:
  \endgraf
  Department of Mathematics
  \endgraf
  Ghent University, Belgium
  \endgraf
  and 
  \endgraf
  School of Mathematics
    \endgraf
    Queen Mary University of London
  \endgraf
  United Kingdom
  \endgraf
  {\it E-mail address} {\rm ruzhansky@gmail.com}
  }

\date{\today}

\thanks{The second author was supported in parts by the FWO Odysseus Project, by the EPSRC Grant EP/R003025/1 and by the Leverhulme Research Grant
RPG-2017-151. }

\begin{abstract}
In this paper we prove  H\"ormander-Mihlin multiplier theorems for pseudo-multipliers associated to the harmonic oscillator (also called the Hermite operator). Our approach can be extended to also obtain the $L^p$-boundedness results for multilinear pseudo-multipliers. By using the Littlewood-Paley theorem associated to the harmonic oscillator  we also give $L^p$-boundedness and  $L^p$-compactness properties for multipliers. $(L^p,L^q)$-estimates for spectral pseudo-multipliers also are investigated.
\end{abstract} \maketitle

\tableofcontents

\section{Introduction}

In this paper we are interested in the $L^p$-boundedness of pseudo-multipliers associated to the harmonic oscillator (also called Hermite pseudo-multipliers) on $L^p(\mathbb{R}^n)$-spaces.
 The harmonic oscillator is the fundamental operator of quantum mechanics defined by
 \begin{equation}
H\psi:=(-\Delta_x+|x|^2)\psi,
\end{equation}
 with $|x|^2:=\sum_{i=1}^nx_i^2.$
 The harmonic oscillator extends to an unbounded self-adjoint operator on $L^{2}(\mathbb{R}^n) $, and its spectrum consists of the discrete set  $\lambda_\nu:=2|\nu|+n,$ $\nu\in \mathbb{N}_0^n,$ with a set of \emph{real eigenfunctions} $\phi_\nu, $ $\nu\in \mathbb{N}_0^n$ (called Hermite functions) which provide an orthonormal basis of ${L}^2(\mathbb{R}^n).$
 Each Hermite function  $\phi_{\nu}$  on $\mathbb{R}^n$ has the form
\begin{equation}
\phi_\nu:=\Pi_{j=1}^n\phi_{\nu_j},\,\,\, \phi_{\nu_j}(x_j):=(2^{\nu_j}\nu_j!\sqrt{\pi})^{-\frac{1}{2}}H_{\nu_j}(x_j)e^{-\frac{1}{2}x_j^2},
\end{equation}
where $x=(x_1,\ldots,x_n)\in\mathbb{R}^n$, $\nu=(\nu_1,\ldots,\nu_n)\in\mathbb{N}^n_0,$ and $$H_{\nu_j}(x_j):=(-1)^{\nu_j}e^{x_j^2}\frac{d^k}{dx_{j}^k}(e^{-x_j^2})$$ denotes the Hermite polynomial of order $\nu_j.$  By the spectral theorem, for every $f\in\mathscr{D}(\mathbb{R}^n)$ we have
\begin{equation}
Hf(x)=\sum_{\nu\in\mathbb{N}^n_0}\lambda_\nu\widehat{f}(\phi_\nu)\phi_\nu(x),\,\,\,
\end{equation} where $\widehat{f}(\phi_\nu) $ is the Fourier-Hermite transform of $f$ at $\nu$ defined by
\begin{equation} (\mathscr{F}_{H}f)(\nu)\equiv\widehat{f}(\phi_\nu) = ( f,\phi_\nu )_{L^2(\mathbb{R}^n)}:=\int_{\mathbb{R}^n}f(x)\phi_\nu(x)\,dx.\end{equation}
A multiplier associated to the harmonic oscillator (or Hermite multiplier) is a linear operator $T_m$ of the form
\begin{equation}
T_mf(x):=\sum_{\nu\in\mathbb{N}^n_0}m(\nu)\widehat{f}(\phi_\nu)\phi_\nu(x),
\end{equation}
for every function $f\in \mathscr{D}(\mathbb{R}^n). $ The discrete function $m$ is called the symbol of the operator $T_m.$ In particular, if  $m$ is a measurable function, the symbol of the spectral multiplier $m(H)$ (defined by the functional calculus) is given by  $m(\nu):=m(\lambda_\nu),$ so that the spectral multipliers are   natural examples of multipliers associated to the harmonic oscillator. We can refer to e.g. Prugove\u{c}ki \cite{Prugovecki} for the quantum mechanical aspects of the harmonic oscillators.

Now, we present some historical results on the analysis of multipliers. If we denote by $P_{\ell}$  the orthogonal projection to the subspace generated by the set $\{\phi_\nu:|\nu|=\ell\},$  and $m$ is a radial function in the sense that $m(\nu)=m(\nu')$ when $|\nu|=|\nu'|,$ then the multiplier $T_m$ can be written as
\begin{equation}\label{radial}
T_{m}\equiv T_{\mu}f(x):=\sum_{\ell=0}^{\infty}\mu(\ell)(P_{\ell}f)(x),
\end{equation}
where $\mu(\vert \nu\vert)=m(\nu).$ An earlier result by G. Mauceri \cite{Mauceri} (by using methods of Bonami-Clerc \cite{Bonami-Clerc73} and R. Strichartz \cite{Strichartz}) states that the condition
\begin{equation}
\sup_{j}2^{j(k-1)}\sum_{2^{j}\leq N\leq 2^{j+1}}|\Delta^{k}_{N}\mu(N)|<\infty,
\end{equation}
where $0\leq k\leq n+1,$ implies the boundedness of $T_\mu$ for all $1<p<\infty.$ As it was pointed out in \cite{thangavelu0}, the number of discrete derivatives $k$  above can be taken in the range $0\leq k\leq [\frac{3n-2}{6}]+2.$ A remarkable result proved by S. Thangavelu (see \cite{Thangavelu}) states that if $m$ satisfies the discrete Marcienkiewicz condition
\begin{equation}\label{thangavelucondition}
|\Delta_\nu^{\alpha} m(\nu)|\leq C_{\alpha}(1+|\nu|)^{-|\alpha|},\,\,\alpha\in\mathbb{N}^n_0,\,|\alpha|\leq [\frac{n}{2}]+1,
\end{equation}
where $\Delta_\nu$ is the usual difference operator, then the corresponding multiplier $T_m:L^p(\mathbb{R}^n)\rightarrow L^p(\mathbb{R}^n)$ extends to a bounded operator for all $1<p<\infty.$
This result is a discrete analogue of the result proved by Mihlin \cite{Mihlin} for Fourier multipliers of the form
\begin{equation}
T_af(x)=\int_{\mathbb{R}^n}a(\xi)\mathscr{F}{f}(\xi)e^{-2\pi i x\cdot \xi}d\xi,
\end{equation}
where $\mathscr{F}$ is the Fourier transform on $\mathbb{R}^n.$ The Mihlin condition states that if $a$ is a function on $\mathbb{R}^{n}$ satisfying
\begin{equation}\label{mihlincondition}
|\partial_\xi^{\alpha} a(\xi)|\leq C_{\alpha}|\xi|^{-|\alpha|}, \,\xi\neq 0, \,\,\alpha\in\mathbb{N}^n_0,\,|\alpha|\leq [\frac{n}{2}]+1,
\end{equation}
then $T_a:L^p(\mathbb{R}^n)\rightarrow L^p(\mathbb{R}^n)$ extends to a bounded operator for all $1<p<\infty.$ In \cite{Hormander1960} H\"ormander generalised the Mihlin condition \eqref{mihlincondition} to the condition of the form
\begin{equation}\label{hormandercondition}
\Vert a\Vert_{l.u.H^s}:=\sup_{r>0}\Vert a(r\cdot)\eta(\cdot) \Vert_{H^s(\mathbb{R}^n)}=\sup_{r>0}r^{s-\frac{n}{2}}\Vert a(\cdot)\eta(r^{-1}\cdot) \Vert_{H^s(\mathbb{R}^n)}<\infty,\,\, \,\,
\end{equation}
where $\eta\in \mathscr{D}(0,\infty)$ and  $s>\frac{n}{2},$ in order to guarantee the boundedness of a Fourier multiplier $T_a$ on $L^p(\mathbb{R}^n)$ for all $1<p<\infty.$

As it was pointed out in \cite{Blunk}, the situation for multipliers associated to the harmonic oscillator is quite different. In fact, for all $s$ and $\varepsilon>0$  with
\begin{equation}\label{critic}
\frac{n}{2}<s\leq \frac{n}{2}+\frac{1}{6}-\varepsilon
\end{equation}
we can not guarantee the $L^p$-boundedness of Riesz means operators satisfying  \eqref{hormandercondition}, for all   $1<p<\infty.$  However, it was proved in  \cite{Blunk} that there exists $p_0\in [1,2]$ such that a general operator $a(H)$ satisfying so-called Plancherel estimates can be extended to a bounded operator on $L^p(\mathbb{R}^n)$ for all $p_0<p<p_0',$ provided that  $a$ satisfies \eqref{hormandercondition} for $s>\frac{n+1}{2}.$ H\"ormander conditions for Hermite operators were established in \cite{SYY}, see also \cite[Theorem III.9]{COSY}.

An extension of Fourier multipliers is given by so-called pseudo-multipliers (see \cite{BagchiThangavelu}). If $m$ is a bounded function on $\mathbb{R}^n\times \mathbb{N}_0^n$ the associated pseudo-multiplier  $T_m$ is the operator defined by
\begin{equation}\label{EQ:pm}
T_{m}f(x):= \sum_{\nu\in\mathbb{N}^n_0}m(x,\nu)\widehat{f}(\phi_\nu)\phi_\nu(x)
\end{equation}
for every function $f\in \mathscr{D}(\mathbb{R}^n). $ We refer to the function $m$ as the \emph{symbol} of the operator $T_m.$ If $m(\nu)=\mu(|\nu|)$ (as in the Maceuri result mentioned previously), it was proved among other things  by  S. Bagchi and S. Thangavelu  \cite{BagchiThangavelu} (see also J. Epperson \cite{Epperson}), that for $n\geq 2,$ the condition
\begin{equation}
\sup_{x\in\mathbb{R}^n}|\Delta^j \mu(x,k)|\leq C_{j}(2k+n)^{-j},\,\,0\leq j\leq n+1,
\end{equation}
implies that the pseudo-multiplier $T_\mu$ is of weak type (1,1) and bounded on $L^p(\mathbb{R}^n)$ provided that $T_{\mu}$ is bounded on $L^2(\mathbb{R}^n).$  The reference \cite{BagchiThangavelu} provides several conditions for the boundedness of pseudo-multipliers including continuity in $L^p$-spaces with weights.\\

From the point of view of the theory of pseudo-differential operators, pseudo-multipliers would be the special case of the symbolic calculus developed in the works of the second author and N. Tokmagambetov \cite{ProfRuzM:TokN:20016,ProfRuzM:TokN:20017}.

The main result of this paper is the H\"ormander type condition for pseudo-multiplier operators \eqref{EQ:pm} and for their multilinear versions. In order to classify the order of regularity in our H\"ormander conditions, we use the following norms,
\begin{equation}\label{primeranorma}
\Vert m \Vert_{l.u., H^s}:=\sup_{k>0,\,y\in\mathbb{R}^n} \,2^{k(s-\frac{n}{2})}\Vert  \langle x\rangle^{s}\mathscr{F}[m(y,\cdot)\psi(2^{-k} |\cdot|)](x)\Vert_{L^2({\mathbb{R}}^n_x)}<\infty,
\end{equation}
\begin{equation}\label{segundanorma}
\Vert m\Vert_{l.u., \mathcal{H}^s}:=\sup_{k>0}\sup_{y\in\mathbb{R}^n} \,{2}^{k(s-\frac{n}{2})}\Vert \langle x \rangle^s \mathscr{F}^{-1}_H[m(y,\cdot)\psi(2^{-k}|\cdot|)](x)\Vert_{L^2(\mathbb{R}^n_x)}<\infty,
\end{equation}
defined by the Fourier transform and the Fourier-Hermite transform, respectively, with $\langle x\rangle:=(1+|x|^2)^{\frac{1}{2}}$. In \eqref{primeranorma} we consider functions $m$ on $\mathbb{R}^n\times \mathbb{R}^n,$ but to these functions we associate a pseudo-multiplier with symbol $\{m(x,\nu)\}_{x\in\mathbb{R}^n,\nu\in\mathbb{N}_0^n}.$ Our main results for pseudo-multipliers can be summarised in the following two theorems.

\begin{theorem}\label{maintheorem}
Let us assume that $2\leq p<\infty.$ If $T_m$ is a pseudo-multiplier with symbol $m$ satisfying  \eqref{primeranorma}, then  under one of the following conditions,
\begin{itemize}
\item  $n\geq 2,$  $2\leq p<\frac{2(n+3)}{n+1},$ and $s>s_{n,p}:=\frac{3n}{2}+{\frac{n-1}{2}(\frac{1}{2}-\frac{1}{p})},$
\item $n\geq 2,$ $p=\frac{2(n+3)}{n+1},$ and $s>s_{n,p}:=\frac{3n}{2}+\frac{n-1}{2(n+3)},$
\item $n\geq 2,$ $\frac{2(n+3)}{n+1}<p\leq \frac{2n}{n-2},$ and  $s>s_{n,p}:=\frac{3n}{2}{-\frac{1}{6}+\frac{2n}{3}(\frac{1}{2}-\frac{1}{p})},$
\item $n\geq 2,$ $\frac{2n}{n-2}\leq p<\infty,$ and  $s>s_{n,p}:=\frac{3n-1}{2}{+n(\frac{1}{2}-\frac{1}{p})},$
\item $n=1,$ $2\leq p<4,$ $s>s_{1,p}:=\frac{3}{2},$
\item $n=1,$ $p=4,$ $s>s_{1,4}:=2,$
\item $n=1,$ $4<p<\infty,$  $s>s_{1,p}:=\frac{4}{3}{+\frac{2}{3}(\frac{1}{2}-\frac{1}{p})},$
\end{itemize}
the operator $T_m$ extends to a bounded operator on $L^p(\mathbb{R}^n).$ For $1<p\leq 2,$ under one of the following  conditions
\begin{itemize}
\item  $n\geq 2,$  $\frac{2(n+3)}{n+5}\leq p\leq 2,$ and $s>s_{n,p}:=\frac{3n}{2}+{\frac{n-1}{2}(\frac{1}{2}-\frac{1}{p})},$
\item $n\geq 2,$ $\frac{2n}{n+2}\leq p\leq \frac{2(n+3)}{n+5},$ and  $s>s_{n,p}:=\frac{3n}{2}{-\frac{1}{6}+\frac{2n}{3}(\frac{1}{2}-\frac{1}{p})},$
\item $n\geq 2,$ $1< p\leq \frac{2n}{n+2},$ and  $s>s_{n,p}:=\frac{3n-1}{2}{+n(\frac{1}{2}-\frac{1}{p})},$
\item $n=1,$ $\frac{4}{3}\leq p<2,$ $s>s_{1,p}:=\frac{3}{2},$
\item $n=1,$ $1<p<\frac{4}{3},$  $s>s_{1,p}:=\frac{4}{3}{+\frac{2}{3}(\frac{1}{2}-\frac{1}{p})},$
\end{itemize} the operator $T_m$ extends to a bounded operator on $L^p(\mathbb{R}^n)$. However, in general:
\begin{itemize} \item for every $\frac{4}{3}<p<4$ and every $n,$ the condition  $s>\frac{3n}{2}$ implies the $L^p$-boundedness of $T_m.$  \end{itemize} If the symbol $m$ of the pseudo-multiplier $T_m$ satisfies the H\"ormander condition \eqref{segundanorma}, in order to guarantee the $L^p$-boundedness of $T_m,$  in every case above we can take $s>s_{n,p}-\frac{1}{12}.$  Moreover, the condition $s>\frac{3n}{2}-\frac{1}{12}$ implies the $L^p$-boundedness of $T_m$ for all $\frac{4}{3}<p<4.$
\end{theorem}
Now we discuss some important facts concerning the results of this paper.
\begin{itemize}
\item It is usual to assume the $L^2$-boundedness of a pseudo-multiplier $T_m$ in order to provide its $L^p$-boundedness (see \cite{BagchiThangavelu} and \cite{Epperson}). Indeed, as it was pointed out in \cite{BagchiThangavelu}, the problem of finding satisfactory conditions for the $L^2$-boundedness of pseudo-multipliers remains open. However, in  our main theorem we solve such problem by considering symbols $m(x,\nu)$ satisfying the H\"ormander condition \eqref{primeranorma} of order $s>\frac{3n}{2},$ uniformly in $y\in \mathbb{R}^n,$ or the  condition \eqref{segundanorma} for $s>\frac{3n}{2}-\frac{1}{12},$ uniformly in $y\in\mathbb{R}^n.$
\item A function $m$ belongs to the Kohn-Nirenberg class $S^{0,\rho}(\mathbb{R}^n\times \mathbb{R}^n)$ if it satisfies the symbol inequalities
\begin{equation}
|\partial_{\xi}^{\alpha}m(x,\xi)|\leq C_{\alpha}(1+|\xi|)^{-|\alpha|},\,\,\,|\alpha|\leq \rho,
\end{equation}
uniformly in $x\in\mathbb{R}^n.$ Symbols in the class $S^{0,2n+1}$ are functions satisfying \eqref{primeranorma} and they provide bounded pseudo-multipliers in $L^p$-spaces for all $1<p<\infty$. In particular symbols in the class $S^{0,\,[\frac{3n}{2}]+1}$ provide  bounded pseudo-multipliers  in $L^2(\mathbb{R}^n).$ These facts will  be proved in Proposition \ref{KNcondition'}. Moreover, (see Corollary \ref{DiscTangCond}) if we assume the condition,
\begin{equation}\label{DiscreteconditionIntro}
|\Delta_\nu^{\alpha} m(x,\nu)|\leq C_{\alpha}(1+|\nu|)^{-|\alpha|},\,\,\alpha\in\mathbb{N}^n_0,\,|\alpha|\leq \rho,
\end{equation}for $\rho=[3n/2]+1,$ then $T_m$ extends to a bounded operator on $L^2(\mathbb{R}^n)$, and for $\rho=2n+1$ we have the $L^p(\mathbb{R}^n)$-boundedness of $T_m$ for all $1<p<\infty.$ 
\item For $n=1$ and by assuming the $L^2$-boundedness of a pseudo-multiplier $T_m$, it was proved by Epperson \cite{Epperson} that \eqref{DiscreteconditionIntro} is a sufficient condition for the $L^p$-boundedness of $T_m$ provided that $\rho=5.$ In constrast, we only require derivatives up to order $\rho=3.$ For spectral pseudo-multipliers $m(x,H)$ and $n\geq 2$,  and newly by assuming the $L^2$-boundedness, Bagchi and Thangavelu proved the $L^p$-boundedness provided that \eqref{DiscreteconditionIntro} holds true for $\rho=n+1.$ Although we impose for $n\geq 2,$ $\rho=[3n/2]+1,$  we do not assume the $L^2$-boundedness for these operators. We also include general pseudo-multipliers  and particularly spectral pseudo-multipliers.
\item The $(L^p,L^q)-$boundedness of pseudo-multipliers will be investigated  in Theorem \ref{LpLp'} and Theorem \ref{Teoremita:Lp:Lq;general}.
\item By using the Littlewood-Paley theorem associated to the harmonic oscillator, we give a $L^p$-multiplier theorem and a $L^p$-compactness theorem for multipliers (see Theorem \ref{LpTLpComp}), the sufficient condition imposed is however, different from the H\"ormander condition. The $L^2$-compactness of multipliers will be characterised in Theorem \ref{CompL2}.
\end{itemize}
In this paper we introduce the notion of multilinear pseudo-multipliers,  which, in analogy with the definition of multilinear Fourier multipliers, are operators of the form
\begin{equation}\label{pseudo-multiplier}
T_{m}(f_1,\cdots, f_\varkappa):=\sum_{\nu:=(\nu_1,\cdots, \nu_\varkappa)\in\mathbb{N}_{0}^{n\varkappa}}m(x,\nu)\widehat{f}_1(\phi_{\nu_1})\cdots \widehat{f}_\varkappa(\phi_{\nu_\varkappa}) \phi_{\nu_1}\cdots \phi_{\nu_\varkappa},\,\,x\in\mathbb{R}^{n},
\end{equation}
for all $f_{1},f_{2},\cdots ,f_\varkappa\in\mathscr{D}(\mathbb{R}^n).$ In this setting, by imposing discrete multilinear H\"ormander conditions on the symbol $m,$ of the type
\begin{equation}\label{multilinearhormander11}
\Vert m \Vert_{l.u.,\mathcal{H}^s}:=\sup_{k>0,\,x\in\mathbb{R}^n} \,2^{k(s-\frac{n\varkappa}{2})}\Vert  \langle z \rangle^{s} \mathscr{F}_{H}^{-1}[m(x,\cdot)\psi(2^{-k}|\cdot|)](z)\Vert_{L^2({\mathbb{R}}^{n\varkappa}_z)}<\infty,
\end{equation}
\begin{equation}\label{multilinearhormander22}
\Vert m \Vert_{l.u.,{H}^s}:=\sup_{k>0,\,x\in\mathbb{R}^n} \,2^{k(s-\frac{n\varkappa}{2})}\Vert  \langle z \rangle^{s} \mathscr{F}[m(x,\cdot)\psi(2^{-k}|\cdot|)](z)\Vert_{L^2({\mathbb{R}}^{n\varkappa}_z)}<\infty,
\end{equation}
we want to guarantee the boundedness of $T_{m} .$ Thus, we establish the  following multilinear result.
\begin{theorem}
Let us consider a multilinear pseudo-multiplier $T_m$ defined on $\mathscr{D}(\mathbb{R}^n)^{\varkappa}$ with symbol satisfying \eqref{multilinearhormander11} or \eqref{multilinearhormander22} for  $$s>s_{n,\varkappa,p}:=\max\{\frac{3n\varkappa}{2}+{(\varkappa-1)}\gamma_\infty,\frac{3n\varkappa}{2}+\frac{(\varkappa-1)n}{4}\},$$ with $\gamma_\infty,$ defined as in \eqref{gammap}. Then the operator
\begin{equation}\label{multi111}
T_m:L^{p_1}\times L^{p_2}\times\cdots \times L^{p_{\varkappa-1}}\times L^{p_\varkappa}\rightarrow L^{p}(\mathbb{R}^n)
\end{equation}
extends to a bounded multilinear operator provided that $1\leq p_{j}\leq \infty,$ $1\leq p\leq 2,$ and $\frac{1}{p}=\frac{1}{p_1}+\cdots +\frac{1}{p_\varkappa}.$ If $m$ satisfies the condition \eqref{multilinearhormander11} or \eqref{multilinearhormander22} for $$s>\max\{\frac{3n\varkappa}{2}+\frac{(\varkappa-1)n}{4},\frac{3n\varkappa}{2}+\frac{(n-1)(\varkappa-1)}{2}+\gamma_p\},$$ with $\gamma_p$ defined as in \eqref{gammap}, then \eqref{multi111} holds true for all $2\leq p\leq \infty$ and $\frac{1}{p}=\frac{1}{p_1}+\cdots +\frac{1}{p_\varkappa}.$
\end{theorem}
Let us note that $\frac{3n\varkappa}{2}+\frac{(\varkappa-1)n}{4} $ and $\frac{3n\varkappa}{2}+\frac{(n-1)(\varkappa-1)}{2}+\gamma_p$ cannot be compared immediately because the sign of $\gamma_p$ depends on the  values of $p.$

This multilinear theorem for pseudo-multipliers  is analogous to ones obtained in the framework of multilinear multipliers. Although the literature for the multilinear analysis is extensive, we refer the reader to \cite{Graf3,Graf2,Graf} and to the seminal  work of R. Coifman and
Y. Meyer  where the multilinear harmonic analysis was originated.

This paper is organised as follows. In Section \ref{mho} we present the proof of our main theorem. In Section \ref{compactnesssection} we discuss the compactness properties. Finally, in Section \ref{multilinear} we prove the result mentioned above for multilinear pseudo-multipliers.

\section{Boundedness of pseudo-multipliers associated to the harmonic oscillator, H\"ormander condition }\label{mho}

Throughout this paper the function  $\psi\in \mathscr{D}(0,\infty)$ will be supported in $[\frac{1}{2},4]$ with $\psi\equiv 1$ on $[1,2].$ In this section we will use functions in a (locally uniformly)  Sobolev space of order $s>0,$ which consists of all functions $m $ on $\mathbb{R}^n\times \mathbb{R}^n$ satisfying
\begin{equation}
\Vert m \Vert_{l.u., H^s}:=\sup_{k>0,\,y\in\mathbb{R}^n} \,2^{k(s-\frac{n}{2})}\Vert  \langle x\rangle^{s}\mathscr{F}[m(y,\cdot)\psi(2^{-k} |\cdot|)](x)\Vert_{L^2({\mathbb{R}}^n_x)}<\infty,
\end{equation}
 in order to establish the $L^p$-boundedness of  Hermite pseudo-multipliers. We have denoted by $\mathscr{F}$  the Fourier transform on $\mathbb{R}^n$ defined by
\begin{equation}
(\mathscr{F}f)(\xi)=\int_{\mathbb{R}^n}e^{-2\pi i x\cdot \xi }f(x)dx.
\end{equation}
Another option that we can use in order to define (local) discrete Sobolev spaces come from  the norm
\begin{equation}
\Vert m\Vert_{l.u., \mathcal{H}^s}:=\sup_{k>0}\sup_{y\in\mathbb{R}^n} \,2^{k(s-\frac{n}{2})}\Vert \langle x \rangle^s \mathscr{F}^{-1}_H[m(y,\cdot)\psi(2^{-k}|\cdot|)](x)\Vert_{L^2(\mathbb{R}^n_x)}<\infty.
\end{equation}
We recall that the Fourier-Hermite transform $\mathscr{F}_{H}$ is defined for every $f\in\mathscr{D}(\mathbb{R}^n)$ by the formula
\begin{equation}\label{FourierHermitetransform}
(\mathscr{F}_Hf)(\nu):=\int_{\mathbb{R}^n}f(x)\psi_\nu(x)dx,\,\,\nu\in\mathbb{N}^n_0.
\end{equation}
If we denote the inverse Fourier-Hermite transform by $\mathscr{F}^{-1}_{H}$ which is defined by
\begin{equation}\label{inverse}
(\mathscr{F}^{-1}_{H}u)(x):=\sum_{\nu\in\mathbb{N}_0^n}u(\nu)\phi_\nu(x),
\end{equation}
where $u$ is a function with compact support on $\mathbb{N}_0^n,$ then the Fourier-Hermite inversion formula is given by
\begin{equation}
f(x)=\sum_{\nu\in\mathbb{N}_0^n}(\mathscr{F}_Hf)(\nu)\phi_\nu(x).
\end{equation}
Now, a pseudo-multiplier $T_m$ with symbol $m$ has, in terms of the transformation $\mathscr{F}_H,$ the alternative representation
\begin{equation}
T_{m}f(x)=\mathscr{F}^{-1}_{H}[m(x,\nu)(\mathscr{F}_Hf)](x).
\end{equation}
For properties and basics of the Fourier-Hermite transform and Hermite  expansions we refer the reader to Thangavelu \cite{Thangavelu}.

\subsection{Hermite functions in $L^p$ spaces}The main tool in the formulation of our results will be estimates of the $L^p$-norms of Hermite functions.  Our starting point is the following lemma for one-dimensional Hermite functions (see Lemma 4.5.2 of Thangavelu \cite{Thangavelu}).
\begin{lemma}\label{Lemma11}
Let us denote by $\phi_{\nu},$ $\nu\in\mathbb{N}_0^n,$ the Hermite functions. As $\nu\rightarrow \infty,$ these functions satisfy the estimates
\begin{itemize}
\item $\Vert \phi_\nu\Vert_{L^p(\mathbb{R})}\asymp  \nu^{\frac{1}{2p}-\frac{1}{4}},$  $1\leq p<4.$
\item $\Vert \phi_\nu\Vert_{L^4(\mathbb{R})}\asymp \nu^{-\frac{1}{8}}\ln(\nu).$
\item $\Vert \phi_\nu\Vert_{L^p(\mathbb{R})}\asymp  \nu^{-\frac{1}{6p}-\frac{1}{12}},$ $4<p\leq \infty.$  \end{itemize}
\end{lemma}
Now, we present a lemma on the behaviour of $L^p(\mathbb{R}^n)$- norms of Hermite functions on $\mathbb{R}^n$ for all $1\leq p\leq 2.$
\begin{lemma}\label{Lemma10}
Let $\phi_\nu,$ $\nu\in\mathbb{N}_0^n,$ be  Hermite functions on $\mathbb{R}^n.$ Then for $1\leq p\leq 2$ we have
\begin{eqnarray}
\Vert \phi_{\nu} \Vert_{L^p(\mathbb{R}^n)}\lesssim |\nu|^{\frac{n}{2}(\frac{1}{p}-\frac{1}{2})}.
\end{eqnarray}
\end{lemma}

\begin{proof}
We will use the first equivalence in Lemma \ref{Lemma11}. Every Hermite function on $\mathbb{R}^n$ has the form $\phi_\nu=\phi_{\nu_1}\times \cdots \times \phi_{\nu_n} $ and as a consequence we have
\begin{align}
\Vert \phi_\nu\Vert_{L^p(\mathbb{R}^n)}=\prod_{j}\Vert \phi_{\nu_j}\Vert_{L^p(\mathbb{R})}.
\end{align}
Now, if $1\leq p\leq 2$ then $\frac{1}{2p}-\frac{1}{4}\geq 0$ and
\begin{align}
\Vert \phi_\nu\Vert_{L^p(\mathbb{R}^n)}\asymp \left(\prod_{j}|\nu_j|\right)^{\frac{1}{2p}-\frac{1}{4}}\leq \left(\frac{\sum_{j}|\nu_j|}{n}\right)^{n(\frac{1}{2p}-\frac{1}{4})}\lesssim |\nu|^{\frac{n}{2}(\frac{1}{p}-\frac{1}{2})},
\end{align}
where  we have used the inequality $x_1\times \cdots \times x_n\leq (\frac{x_1+\cdots+ x_n}{n})^{n}$  for  $x_i>0.$
\end{proof}
We now recall the following sharp lemma on the $L^p$-norms of Hermite functions for $2\leq p\leq \infty$ (see H. Koch and D. Tataru \cite{Koch}).
\begin{lemma}\label{tataru}
Let us consider a Hermite function $\phi=\phi_\nu$ on $\mathbb{R}^n$ which, as an eigenfunction of the harmonic oscillator on $\mathbb{R}^n,$ has the associated eigenvalue $\lambda^2=(2|\nu|+n).$ Then for $n\geq 2$ we have,
\begin{itemize}
\item if $2\leq p<\frac{2(n+3)}{n+1},$ then $\Vert \phi_\nu\Vert_{L^p(\mathbb{R}^n)}\lesssim (2|\nu|+n)^{\frac{1}{2p}-\frac{1}{4}},$
\item if $\frac{2(n+3)}{n+1}<p\leq \frac{2n}{n-2},$ then $\Vert \phi_\nu\Vert_{L^p(\mathbb{R}^n)}\lesssim (2|\nu|+n)^{-\frac{1}{6}+\frac{n}{6}(\frac{1}{2}-\frac{1}{p})},$
\item if $\frac{2n}{n-2}\leq p\leq \infty,$ then $\Vert \phi_\nu\Vert_{L^p(\mathbb{R}^n)}\lesssim (2|\nu|+n)^{-\frac{1}{2}+\frac{n}{2}(\frac{1}{2}-\frac{1}{p})},$
\end{itemize} and for $n=1,$
\begin{itemize}
\item if $2\leq p<4,$ $\Vert \phi_\nu\Vert_{L^p(\mathbb{R})}\lesssim (2\nu+n)^{-\frac{1}{2}(\frac{1}{2}-\frac{1}{p})}, $
\item if $4<p\leq \infty,$ $\Vert \phi_\nu\Vert_{L^p(\mathbb{R})}\lesssim(2\nu+n)^{-\frac{1}{6}+\frac{1}{6}(\frac{1}{p}-\frac{1}{2})}.$
\end{itemize}
\end{lemma}
It is important to mention that in the previous lemma we denote $\frac{2n}{n-2}=\infty,$ when $n=2.$ We adopt this convention in the whole paper. Let us mention that, curiously, the proof of the lemma above is a consequence of some dispersive
and Strichartz estimates for the corresponding Schr\"odinger equation for the harmonic oscillator. In our further analysis, we will  need the following lemma.
\begin{lemma}\label{Lemma1}
Let us assume that $2\leq p\leq \infty$ and $n\geq 2.$ Then, the Hermite functions satisfy the following estimates as $|\nu|\rightarrow \infty:$
\begin{itemize}
\item if $2\leq p<\frac{2(n+3)}{n+1},$ then $$\Vert \phi_\nu\Vert_{L^p(\mathbb{R}^n)}\Vert \phi_\nu\Vert_{L^{p'}(\mathbb{R}^n)}\lesssim |\nu|^{\frac{n-1}{2}(\frac{1}{2}-\frac{1}{p})},$$
\item if $\frac{2(n+3)}{n+1}<p\leq \frac{2n}{n-2},$ then $$\Vert \phi_\nu\Vert_{L^p(\mathbb{R}^n)}\Vert \phi_\nu\Vert_{L^{p'}(\mathbb{R}^n)}\lesssim |\nu|^{-\frac{1}{6}+\frac{2n}{3}(\frac{1}{2}-\frac{1}{p})},$$
\item if $\frac{2n}{n-2}\leq p\leq \infty,$ then $$\Vert \phi_\nu\Vert_{L^p(\mathbb{R}^n)}\Vert \phi_\nu\Vert_{L^{p'}(\mathbb{R}^n)}\lesssim |\nu|^{-\frac{1}{2}+n(\frac{1}{2}-\frac{1}{p})}.$$
\end{itemize}
Let us recall that we have denoted $\frac{2n}{n-2}=\infty$ when $n=2.$ For $n=1$ we have
\begin{itemize}
\item if $2\leq p<4,$ $$\Vert \phi_\nu\Vert_{L^p(\mathbb{R})}\Vert \phi_\nu\Vert_{L^{p'}(\mathbb{R})}\lesssim 1,$$
\item if $4<p\leq \infty,$  $$\Vert \phi_\nu\Vert_{L^p(\mathbb{R})}\Vert \phi_\nu\Vert_{L^{p'}(\mathbb{R})}\lesssim \nu^{-\frac{1}{6}+\frac{2}{3}(\frac{1}{2}-\frac{1}{p})}.$$
\end{itemize}
In general:
\begin{itemize} \item for every $\frac{4}{3}<p<4$ and every $n,$  $\Vert \phi_\nu\Vert_{L^p(\mathbb{R}^n)}\Vert \phi_\nu\Vert_{L^{p'}(\mathbb{R}^n)}=O(1).$ \end{itemize}
\end{lemma}
\begin{proof}
Except for the last item, the proof is a straightforward computation by replacing $p$ in Lemma  \ref{Lemma10} and the estimates in Lemma \ref{tataru}. The last item can be proved by using that $p,p'\in(\frac{4}{3},4)$ and the first estimate in Lemma \ref{Lemma11}, in fact
\begin{align*}
\Vert \phi_\nu\Vert_{L^p(\mathbb{R}^n)}\Vert \phi_\nu\Vert_{L^{p'}(\mathbb{R}^n)} &=\prod_{j}\Vert \phi_{\nu_j}\Vert_{L^p(\mathbb{R})}\Vert \phi_{\nu_j}\Vert_{L^{p'}(\mathbb{R})}\\
&\asymp \prod_{j}\nu_j^{\frac{1}{2p}-\frac{1}{4}}\nu_j^{\frac{1}{2p'}-\frac{1}{4}}= \prod_{j}\nu_j^{\frac{1}{2}(\frac{1}{p}+\frac{1}{p'})-\frac{1}{2}}= \prod_{j} 1=1,
\end{align*} completing the proof.
\end{proof}
\begin{remark}\label{infty}
Because $\Vert \phi_{\nu_j}\Vert_{L^\infty(\mathbb{R})}\lesssim |\nu_j|^{-\frac{1}{12}}$ when $\nu_j\rightarrow\infty,$ we can estimate $\Vert \phi_\nu\Vert_{L^\infty(\mathbb{R}^n)}\lesssim |\nu|^{-\frac{1}{12}}.$ Indeed, when $|\nu|\rightarrow \infty,$ then $\nu_i:=\max_{1\leq j\leq n}\nu_j\rightarrow \infty,$  and from the inequality $|\nu|\leq n\nu_i$ we obtain $\nu_i^{-\frac{1}{12}}\leq n^{\frac{1}{12}}|\nu|^{-\frac{1}{12}}$ which implies the desired estimate.
\end{remark}

\subsection{H\"ormander condition for pseudo-multipliers on $L^p$ spaces}

Now, we analyse the boundedness of pseudo-multipliers with symbols in (locally uniform)  Sobolev spaces. We denote by $\gamma_p$ the exponent that according to Lemma \ref{Lemma1} satisfies
\begin{equation}\label{gammap}
\Vert \phi_\nu\Vert_{L^{p}(\mathbb{R}^n)} \Vert  \phi_\nu\Vert_{L^{p'}(\mathbb{R}^n)}\lesssim |\nu|^{\gamma_p}.
\end{equation}

\begin{remark}\label{lowerboundforgammap}
Since
\begin{itemize}
\item $n \geq 2,$ $2\leq p<\frac{2(n+3)}{n+1},$ implies $0\leq \gamma_p:={\frac{n-1}{2}(\frac{1}{2}-\frac{1}{p})}< \frac{n-1}{2(n+3)},$
\item $n \geq 2,$ $\frac{2(n+3)}{n+1}<p\leq \frac{2n}{n-2},$ implies $ -\frac{1}{6}+\frac{2n}{3(n+3)}\leq \gamma_p:={-\frac{1}{6}+\frac{2n}{3}(\frac{1}{2}-\frac{1}{p})}<\frac{1}{2} ,$
\item $n \geq 2,$ $\frac{2n}{n-2}\leq p\leq \infty,$ implies $ \frac{1}{2}\leq \gamma_p:={-\frac{1}{2}+n(\frac{1}{2}-\frac{1}{p})}\leq \frac{n-1}{2},$
\item $n=1,$ $2\leq p<4,$ implies $\gamma_p=0,$
\item $n=1,$ $4<p\leq \infty,$ implies $\frac{1}{4}<\gamma_p:=\frac{1}{2}-\frac{1}{p}\leq \frac{1}{2},$
\end{itemize}
we have that $\gamma_p\geq 0,$ for all $2\leq p\leq \infty.$ This lower bound will be useful in our further analysis.
\end{remark}
\begin{proposition}\label{proposition1}
Let us consider $1<p<\infty$ and $s>\frac{3n}{2}+\gamma_p-\frac{1}{12}.$ If $T_m$ is a pseudo-multiplier with symbol $m$ satisfying
\begin{eqnarray}\label{FHcondition}
\Vert m \Vert_{l.u.\mathcal{H}^s}=\sup_{k>0,x\in\mathbb{R}^n}2^{k(s-\frac{n}{2})}\Vert \langle \,\cdot\,\rangle^{s}\mathscr{F}_{H}^{-1}[m(x,\cdot)\psi({2^{-k}|\cdot|})]\Vert_{L^2(\mathbb{R}^n)}<\infty,
\end{eqnarray} then $T_{m}$ extends to a bounded operator on $L^p(\mathbb{R}^n).$ Moreover,
\begin{equation}\Vert T_{m} \Vert_{\mathscr{B}(L^p(\mathbb{R}^n))}\leq  C (\Vert m \Vert_{l.u.,\mathcal{H}^s}+\Vert m(\cdot,0)\Vert_{L^\infty(\mathbb{R}^n)}). \end{equation}
\end{proposition}
\begin{proof}
In order to prove  Proposition \ref{proposition1} we will decompose the symbol $m$ as
 \begin{equation}
  m(x,\nu)=m(x,0)+\sum_{k=0}^{\infty}  m_k(x,\nu),\,\,\,\,\, m_k(x,\nu):= m(x,\nu)\cdot 1_{\{2^{k}\leq |\nu|<2^{k+1}\}}.
 \end{equation}
Let us denote by $T_{m(j)}  $ the pseudo-multiplier associated to $m_j,$  for $j\geq 0,$ and by $T_{0} $ the operator with symbol   $\sigma\equiv m(x,0)\delta_{\nu,0}.$ Then we want to show that the operator series
\begin{equation}
T_0+\sum_{k} T_{m(k)}
\end{equation}
converges to $T_m$ in the strong topology on $\mathscr{B}(L^p(\mathbb{R}^n))$ and
\begin{equation}
\Vert T_m \Vert_{\mathscr{B}(L^p(\mathbb{R}^n))}  \leq \Vert T_0 \Vert_{\mathscr{B}(L^p(\mathbb{R}^n))} +\sum_k  \Vert T_{m(k)} \Vert_{\mathscr{B}(L^p(\mathbb{R}^n))} .
\end{equation}

So, we want to estimate every norm $\Vert T_{m(j)} \Vert_{\mathscr{B}(L^p(\mathbb{R}^n))}.$
For this, we will use the fact that for $f\in C^\infty_{0}(\mathbb{R}^n),$
\begin{equation}
\Vert T_{m(j)}f \Vert_{L^p(\mathbb{R}^n)}=\sup\{ ( T_{m(j)} f,g)_{L^2(\mathbb{R}^n)}\, :\, \Vert g\Vert_{L^{p'}(\mathbb{R}^n)}=1 \}.
\end{equation}
In fact, for $f$ and $g$ as above we have
\begin{align*}
( T_{m(k)} f,\overline{g})_{L^2(\mathbb{R}^n)} &=\int_{\mathbb{R}^n}T_{m(k)}f(x)g(x)dx\\
&=\int_{\mathbb{R}^n}\sum_{2^k\leq|\nu|<2^{k+1}}m(x,\nu)\widehat{f}(\phi_\nu)\phi_\nu(x)g(x)dx\\
&=\int_{\mathbb{R}^n}\int_{\mathbb{R}^n}\sum_{2^k\leq|\nu|<2^{k+1}}m(x,\nu)f(y)\phi_\nu(y)\phi_\nu(x)g(x)dydx.
\end{align*}

For every $x\in\mathbb{R}^n$ let us denote the inverse Fourier-Hermite transform of the sequence $\{m(x,\nu)\psi(2^{-k}|\nu|)\}_{\nu}$ by $\mathscr{F}^{-1}_{H}[m(x,\cdot)\psi(2^{-k}|\cdot|)].$ So, we have
\begin{equation}
m_k(x,\nu)=\mathscr{F}_{H}(\mathscr{F}_{H}^{-1}[m(x,\cdot)\psi(2^{-k}|\cdot|)])(\nu)=\int_{\mathbb{R}^n} \mathscr{F}_{H}^{-1}[m(x,\cdot)\psi(2^{-k}|\cdot|)](z)\phi_{\nu}(z)dz.
\end{equation}
Consequently, we can write
\begin{align*}
&( T_{m(k)} f,\overline{g})_{L^2(\mathbb{R}^n)} =\\
&\int_{\mathbb{R}^n}\int_{\mathbb{R}^n}\sum_{2^k\leq|\nu|<2^{k+1}}\int_{\mathbb{R}^n} \mathscr{F}_{H}^{-1}[m(x,\cdot)\psi(2^{-k}|\cdot|)](z)\phi_{\nu}(z)dz \\
&\hspace{8cm}\times f(y)\phi_\nu(y)\phi_\nu(x)g(x)dydx.
\end{align*}
Now, we have
\begin{align*}
&|( T_{m(k)} f,\overline{g})_{L^2(\mathbb{R}^n)}|\\
&\leq \sum_{2^k\leq|\nu|<2^{k+1}}\sup_{x\in\mathbb{R}^n}\int_{\mathbb{R}^n} |\mathscr{F}_{H}^{-1}[m(x,\cdot)\psi(2^{-k}|\cdot|)](z)||\phi_{\nu}(z)|dz \\
&\hspace{8cm}\times \Vert f\Vert_{L^p}\Vert g\Vert_{L^{p'}}\Vert \phi_\nu\Vert_{L^{p}}\Vert \Vert  \phi_\nu\Vert_{L^{p'}}\\
&\lesssim \sum_{2^k\leq|\nu|<2^{k+1}}\sup_{x\in\mathbb{R}^n}\int_{\mathbb{R}^n} |\mathscr{F}_{H}^{-1}[m(x,\cdot)\psi(2^{-k}|\cdot|)](z)||\phi_{\nu}(z)|dz\\
&\hspace{8cm}\times\Vert f\Vert_{L^p}\Vert g\Vert_{L^{p'}}|\nu|^{\gamma_p}.
\end{align*}
So, we can estimate the operator norm of $T_{m(k)}$ by
\begin{align*}
&\Vert T_{m(k)} \Vert_{\mathscr{B}(L^p)}\\
&\lesssim \sum_{2^k\leq|\nu|<2^{k+1}}\sup_{x\in\mathbb{R}^n}\int_{\mathbb{R}^n} |\mathscr{F}_{H}^{-1}[m(x,\cdot)\psi(2^{-k}|\cdot|)](z)||\phi_{\nu}(z)|dz|\nu|^{\gamma_p}\\
&\lesssim \sum_{2^k\leq|\nu|<2^{k+1}}\sup_{x\in\mathbb{R}^n}\left(\int_{\mathbb{R}^n} \langle z\rangle^{2s}|\mathscr{F}_{H}^{-1}[m(x,\cdot)\psi(2^{-k}|\cdot|)](z)|^{2} dz\right)^{\frac{1}{2}}\Vert \phi_\nu(\cdot)\langle \cdot\rangle^{-s}\Vert_{L^2}|\nu|^{\gamma_p}.
\end{align*}

If we denote by $\theta_\infty$ some  real number satisfying $\Vert \phi_\nu \Vert_{L^{\infty}(\mathbb{R}^n)}\lesssim |\nu|^{\theta_\infty}, $ we can estimate $$\Vert \phi_\nu(\cdot)\langle \cdot\rangle^{-s}\Vert_{L^2}\leq \Vert \phi_{\nu}\Vert_{L^\infty}\Vert \langle \cdot\rangle^{-s}\Vert_{L^2}\lesssim |\nu|^{\theta_{\infty}},$$ if we require $s>\frac{n}{2}.$ By Remark \ref{lowerboundforgammap}, the  condition $s>\frac{n}{2}$ holds true because $s>\frac{3n}{2}+\gamma_p-\frac{1}{12}\geq \frac{3n}{2}-\frac{1}{12}>\frac{n}{2}.$ Now, if additionally we consider the hypothesis
\begin{equation}
\sup_{x\in\mathbb{R}^n}\left(\int_{\mathbb{R}^n} \langle z\rangle^{2s}|\mathscr{F}_{H}^{-1}[m(x,\cdot)\psi(2^{-k}|\cdot|)](z)|^{2} dz\right)^{\frac{1}{2}}\leq \Vert m\Vert_{l.u.\mathcal{H}^s}\cdot 2^{-k(s-\frac{n}{2})},
\end{equation}
then we have
\begin{align*}
\Vert T_{m(k)} \Vert_{\mathscr{B}(L^p)} &\lesssim \sum_{2^k\leq|\nu|<2^{k+1}}\Vert m\Vert_{l.u.\mathcal{H}^s}\cdot 2^{-k(s-\frac{n}{2})}|\nu|^{\gamma_p+\theta_\infty}\\
&\lesssim  2^{kn-k(s-\frac{n}{2})+k\gamma_p+k\theta_\infty}\Vert m\Vert_{l.u.\mathcal{H}^s}=2^{-k(s-\frac{3n}{2}-\gamma_p-\theta_\infty)}.
\end{align*}
Taking into account that
\begin{align*}
\Vert T_{0}f \Vert_{L^p(\mathbb{R}^n)}\lesssim \Vert m(\cdot,0)\Vert_{L^\infty(\mathbb{R}^n)}\Vert  f \Vert_{L^p(\mathbb{R}^n)},\\
\end{align*}
we obtain the boundedness of $T_0$ on $L^p.$ It is clear that if we  want to end the proof, we need  to estimate  $I:=\sum_{k\geq 0} \Vert T_{m(k)} \Vert_{\mathscr{B}(L^p(\mathbb{R}^n))}.$
As a consequence we obtain  $$0\leq I\lesssim \Vert T_{0} \Vert_{\mathscr{B}(L^p)}+ \sum_{k=1}^\infty2^{-k(s-\frac{3n}{2}-\gamma_p-\theta_\infty)} \Vert m\Vert_{l.u., \mathcal{H}^s}<\infty,$$ for $s>\frac{3n}{2}+\gamma_p+\theta_\infty.$ So, we have
$$\Vert T_{m} \Vert_{\mathscr{B}(L^p)}\leq  C (\Vert m \Vert_{l.u.,\mathcal{H}^s}+\Vert m(\cdot,0)\Vert_{L^\infty(\mathbb{R}^n)}). $$ From Remark \ref{infty} we end the proof because we can take $\theta_\infty=-\frac{1}{12}.$
\end{proof}
\begin{proposition}\label{PFcondition}
Let us consider $1<p<\infty$ and $s>\frac{3n}{2}+\gamma_p.$ If $m:\mathbb{R}^{2n}\rightarrow \mathbb{C}$ is a function, and  $T_m$ is a pseudo-multiplier with symbol $ \{m(x,\nu)\}_{x\in\mathbb{R}^n,\nu\in\mathbb{N}_0^n}$ satisfying
\begin{eqnarray}\label{Fcondition}
\Vert m \Vert_{l.u.{H}^s}=\sup_{k>0,x\in\mathbb{R}^n}2^{k(s-\frac{n}{2})}\Vert \langle \,\cdot\,\rangle^{s}\mathscr{F}[m(x,\cdot)\psi({2^{-k}|\cdot|})]\Vert_{L^2(\mathbb{R}^n)}<\infty,
\end{eqnarray} then $T_{m}$ extends to a bounded operator on $L^p(\mathbb{R}^n).$ Moreover,
\begin{equation}\Vert T_{m} \Vert_{\mathscr{B}(L^p)}\leq  C (\Vert m \Vert_{l.u.,{H}^s}+\Vert m(\cdot,0)\Vert_{L^\infty(\mathbb{R}^n)}). \end{equation}
\end{proposition}
\begin{proof}
By following the notation in the proof of Proposition \ref{proposition1} we have
\begin{align*}
( T_{m(k)} f,\overline{g})_{L^2(\mathbb{R}^n)}
&=\int_{\mathbb{R}^n}\int_{\mathbb{R}^n}\sum_{2^k\leq|\nu|<2^{k+1}}m(x,\nu)f(y)\phi_\nu(y)\phi_\nu(x)g(x)dydx.
\end{align*}
For every $x\in\mathbb{R}^n$ let us write
\begin{equation}
m_k(x,\nu)=\mathscr{F}^{-1}(\mathscr{F}[m(x,\cdot)\psi(2^{-k}|\cdot|)])(\nu)=\int_{\mathbb{R}^n} \mathscr{F}[m(x,\cdot)\psi(2^{-k}|\cdot|)](z)e^{2\pi i \nu\cdot z}dz.
\end{equation} So, we have
\begin{align*}
&|( T_{m(k)} f,\overline{g})_{L^2(\mathbb{R}^n)}|\\
&\leq \sum_{2^k\leq|\nu|<2^{k+1}}\sup_{x\in\mathbb{R}^n}\int_{\mathbb{R}^n} |\mathscr{F}[m(x,\cdot)\psi(2^{-k}|\cdot|)](z)|dz \\
&\hspace{8cm}\times \Vert f\Vert_{L^p}\Vert g\Vert_{L^{p'}}\Vert \phi_\nu\Vert_{L^{p}}\Vert \Vert  \phi_\nu\Vert_{L^{p'}}\\
&\lesssim \sum_{2^k\leq|\nu|<2^{k+1}}\sup_{x\in\mathbb{R}^n}\int_{\mathbb{R}^n} |\mathscr{F}[m(x,\cdot)\psi(2^{-k}|\cdot|)](z)|dz\\
&\hspace{8cm}\times\Vert f\Vert_{L^p}\Vert g\Vert_{L^{p'}}|\nu|^{\gamma_p}.
\end{align*} So, we can estimate the operator norm of $T_{m(k)}$ by
\begin{align*}
&\Vert T_{m(k)} \Vert_{\mathscr{B}(L^p)}\\
&\lesssim \sum_{2^k\leq|\nu|<2^{k+1}}\sup_{x\in\mathbb{R}^n}\int_{\mathbb{R}^n} |\mathscr{F}[m(x,\cdot)\psi(2^{-k}|\cdot|)](z)|dz|\nu|^{\gamma_p}\\
&\lesssim \sum_{2^k\leq|\nu|<2^{k+1}}\sup_{x\in\mathbb{R}^n}\left(\int_{\mathbb{R}^n} \langle z\rangle^{2s}|\mathscr{F}[m(x,\cdot)\psi(2^{-k}|\cdot|)](z)|^{2} dz\right)^{\frac{1}{2}}\Vert\langle\, \cdot\,\rangle^{-s}\Vert_{L^2}|\nu|^{\gamma_p}.
\end{align*}
By Remark \ref{lowerboundforgammap} we have  $s>\frac{n}{2},$ together with the estimate $\Vert\langle\, \cdot\,\rangle^{-s}\Vert_{L^2}<\infty,$ and by the hypothesis
\begin{equation}
\sup_{x\in\mathbb{R}^n}\left(\int_{\mathbb{R}^n} \langle z\rangle^{2s}|\mathscr{F}[m(x,\cdot)\psi(2^{-k}|\cdot|)](z)|^{2} dz\right)^{\frac{1}{2}}\leq \Vert m\Vert_{l.u.H^s}\cdot 2^{-k(s-\frac{n}{2})},
\end{equation}
we deduce that
\begin{align*}
\Vert T_{m(k)} \Vert_{\mathscr{B}(L^p)} &\lesssim \sum_{2^k\leq|\nu|<2^{k+1}}\Vert m\Vert_{l.u.H^s}\cdot 2^{-k(s-\frac{n}{2})}|\nu|^{\gamma_p}\\
&\asymp 2^{kn-k(s-\frac{n}{2})+k\gamma_p}=2^{-k(s-\frac{3n}{2}-\gamma_p)}.
\end{align*}
Since
\begin{align*}
\Vert T_{0}f \Vert_{L^p(\mathbb{R}^n)}\lesssim \Vert m(\cdot,0)\Vert_{L^\infty(\mathbb{R}^n)}\Vert  f \Vert_{L^p(\mathbb{R}^n)},\\
\end{align*}
we have the boundedness of $T_0$ on $L^p.$ It is clear that if we  want to end the proof, we need  to estimate  $I:=\sum_{k\geq 0} \Vert T_{m(k)} \Vert_{\mathscr{B}(L^p(\mathbb{R}^n))}.$
As a consequence we obtain  $$0<I\lesssim \Vert T_{0} \Vert_{\mathscr{B}(L^p)}+ \sum_{k=1}^\infty2^{-k(s-\frac{3n}{2}-\gamma_p)}\sup_{x\in\mathbb{R}^n} \Vert m\Vert_{l.u., H^s}<\infty,$$ for $s>\frac{3n}{2}+\gamma_p.$ So, we have
$$\Vert T_{m} \Vert_{\mathscr{B}(L^p)}\leq  C (\Vert m(x,\cdot) \Vert_{l.u.,H^s}+\Vert m(\cdot,0)\Vert_{L^\infty(\mathbb{R}^n)}). $$ The proof is complete.
\end{proof}

Now, we record explicitly  the degree of regularity $s$ considered in the propositions above.
\begin{theorem}\label{TheoremAnnoucenment}
Let us assume $2\leq p<\infty.$ If $T_m$ is a pseudo-multiplier with symbol $m$ satisfying  \eqref{Fcondition}, then  under one of the following conditions,
\begin{itemize}
\item  $n\geq 2,$  $2\leq p<\frac{2(n+3)}{n+1},$ and $s>s_{n,p}:=\frac{3n}{2}+{\frac{n-1}{2}(\frac{1}{2}-\frac{1}{p})},$
\item $n\geq 2,$ $p=\frac{2(n+3)}{n+1},$ and $s>s_{n,p}:=\frac{3n}{2}+\frac{n-1}{2(n+3)},$
\item $n\geq 2,$ $\frac{2(n+3)}{n+1}<p\leq \frac{2n}{n-2},$ and  $s>s_{n,p}:=\frac{3n}{2}{-\frac{1}{6}+\frac{2n}{3}(\frac{1}{2}-\frac{1}{p})},$
\item $n\geq 2,$ $\frac{2n}{n-2}\leq p<\infty,$ and  $s>s_{n,p}:=\frac{3n-1}{2}{+n(\frac{1}{2}-\frac{1}{p})},$
\item $n=1,$ $2\leq p<4,$ $s>s_{1,p}:=\frac{3}{2},$
\item $n=1,$ $p=4,$ $s>s_{1,4}:=2,$
\item $n=1,$ $4<p<\infty,$  $s>s_{1,p}:=\frac{4}{3}{+\frac{2}{3}(\frac{1}{2}-\frac{1}{p})},$
\end{itemize}
the operator $T_m$ extends to a bounded operator on $L^p(\mathbb{R}^n).$ For $1<p\leq 2,$ under one of the following  conditions
\begin{itemize}
\item  $n\geq 2,$  $\frac{2(n+3)}{n+5}\leq p\leq 2,$ and $s>s_{n,p}:=\frac{3n}{2}+{\frac{n-1}{2}(\frac{1}{2}-\frac{1}{p})},$
\item $n\geq 2,$ $\frac{2n}{n+2}\leq p\leq \frac{2(n+3)}{n+5},$ and  $s>s_{n,p}:=\frac{3n}{2}{-\frac{1}{6}+\frac{2n}{3}(\frac{1}{2}-\frac{1}{p})},$
\item $n\geq 2,$ $1< p\leq \frac{2n}{n+2},$ and  $s>s_{n,p}:=\frac{3n-1}{2}{+n(\frac{1}{2}-\frac{1}{p})},$
\item $n=1,$ $\frac{4}{3}\leq p<2,$ $s>s_{1,p}:=\frac{3}{2},$
\item $n=1,$ $1<p<\frac{4}{3},$  $s>s_{1,p}:=\frac{4}{3}{+\frac{2}{3}(\frac{1}{2}-\frac{1}{p})},$
\end{itemize} the operator $T_m$ is $L^p$-bounded. Moreover, in general:
\begin{itemize} \item for every $\frac{4}{3}<p<4$ and every $n,$ the condition  $s>\frac{3n}{2}$ implies the $L^p$-boundedness of $T_m.$  \end{itemize} If the symbol $m$ of the pseudo-multiplier $T_m$ satisfies the H\"ormander condition \eqref{FHcondition}, in order to guarantee the $L^p$-boundedness of $T_m,$  in every case above we can take $s>s_{n,p}-\frac{1}{12}.$  However, $s>\frac{3n}{2}-\frac{1}{12}$ implies the $L^p$-boundedness of $T_m$ for all $\frac{4}{3}<p<4.$
\end{theorem}
\begin{proof}
In view of the Propositions \ref{proposition1} and \ref{PFcondition} and considering the following values for $\gamma_p$: (according to Lemma \ref{Lemma1}),
\begin{itemize}
\item $\gamma_{p}={\frac{n-1}{2}(\frac{1}{2}-\frac{1}{p})},$ if $n\geq 2,$ $2\leq p<\frac{2(n+3)}{n+1},$
\item $\gamma_{p}={-\frac{1}{6}+\frac{2n}{3}(\frac{1}{2}-\frac{1}{p})},$ if $n\geq 2,$ $\frac{2(n+3)}{n+1}<p\leq \frac{2n}{n-2},$
\item $\gamma_p={-\frac{1}{2}+n(\frac{1}{2}-\frac{1}{p})},$ if  $n\geq 2,$ $\frac{2n}{n-2}\leq p< \infty,$
\item $\gamma_{p}=0,$ if $n\in\mathbb{N},$ $\frac{4}{3}< p<4,$
\item $\gamma_p={-\frac{1}{6}+\frac{2}{3}(\frac{1}{2}-\frac{1}{p})},$ if $n=1,$ $4<p<\infty,$
\end{itemize} the proof ends if we take into account that $\gamma_p=\gamma_{p'}$ and
\begin{itemize}
\item $n\geq 2,$ $\frac{2(n+3)}{n+5}\leq p\leq 2$ $\Rightarrow$ $2\leq p'<\frac{2(n+3)}{n+1},$
\item $n\geq 2,$ $\frac{2(n+3)}{n+1}<p\leq \frac{2n}{n-2}$ $\Rightarrow$ $\frac{2n}{n+2}\leq p'\leq \frac{2(n+3)}{n+5},$
\item $n\geq 2,$ $\frac{2n}{n-2}\leq p< \infty$ $\Rightarrow$   $1< p'\leq \frac{2n}{n+2},$
\item $\frac{4}{3}\leq p\leq 2$ $\Rightarrow$  $2\leq p'\leq 4,$
\item $1<p<\frac{4}{3}$ $\Rightarrow$ $4< p'<\infty,$
\item $\frac{4}{3}<p<4$ $\Leftrightarrow$ $\frac{4}{3}<p'<4 .$
\end{itemize} The proof is complete.
\end{proof}
\begin{remark}
 Let us note that for $n\geq 2$ and $p=\frac{2(n+3)}{n+1},$  the condition $s>s_{n,p}=\frac{3n}{2}+\frac{n-1}{2(n+3)},$ implies the boundedness of a pseudo-multiplier $T_m$ on $ L^{ \frac{2(n+3)}{n+1} }(\mathbb{R}^n)$ provided that $m$ satisfies \eqref{Fcondition}. Indeed, from  Remark \ref{lowerboundforgammap}, if $2\leq r<\frac{2(n+3)}{n+1}<\omega\leq \frac{2n}{n-2},$ then 
\begin{equation}
\lim_{  r \nearrow \frac{2(n+3)}{n+1}}\gamma_r   =\lim_{  \omega \searrow   \frac{2(n+3)}{n+1}}\gamma_\omega=\frac{n-1}{2(n+3)}.
\end{equation} 
 From the real interpolation we obtain $\gamma_{   \frac{2(n+3)}{n+1} }=\frac{n-1}{2(n+3)}.$ So, by Proposition \ref{PFcondition}  the condition $s>s_{n,p}=\frac{3n}{2}+\frac{n-1}{2(n+3)}$ implies the $L^p(\mathbb{R}^n)$-boundedness of $T_m$ when $p=\frac{2(n+3)}{n+1}.$ Now, if $n=1,$ a similar analysis shows that  $\gamma_4<\frac{1}{2}$ and the condition $s>s_{1,4}=2$ implies the boundedness of $T_m$ on $L^4(\mathbb{R}^n).$ So, this remark and Theorem \ref{TheoremAnnoucenment} proves  Theorem \ref{maintheorem}.
\end{remark}

In the following proposition we exhibit a class of symbols providing $L^p$-pseudo-multipliers.
\begin{proposition}\label{KNcondition'} Let us consider a complex-valued function $m$ on $\mathbb{R}^n\times\mathbb{R}^n$ and a pseudo-multiplier $T_m$ with symbol $\{m(x,\nu)\}_{x\in\mathbb{R}^n,\nu\in\mathbb{N}_0^n}.$ If $m$ satisfies the symbol inequalities
\begin{equation}\label{KNcondition}
|\partial_{\xi}^{\alpha}m(x,\xi)|\leq C_{\alpha}(1+|\xi|)^{-|\alpha|},\,\,\,|\alpha|\leq \rho,
\end{equation} for $\rho=[3n/2]+1,$ then $T_m$ extends to a bounded operator on $L^2(\mathbb{R}^n)$. Moreover, for $\rho=2n+1$ we have the $L^p(\mathbb{R}^n)$-boundedness of $T_m$ for all $1<p<\infty.$
\end{proposition} \begin{proof}
For the proof, we will use that the Sobolev  space $H^{s}(\mathbb{R}^n)$ defined by those functions $g$ satisfying $ \Vert g\Vert_{H^s(\mathbb{R}^n)}:=\Vert \langle z\rangle^s(\mathscr{F}g)\Vert_{L^2(\mathbb{R}^n)}<\infty, $
has the equivalent norm
\begin{equation}
\Vert g\Vert'_{H^s(\mathbb{R}^n)}:=\sum_{|\beta|\leq s}\Vert \partial_{\xi}^\beta g \Vert_{L^2(\mathbb{R}^n)},
\end{equation}
when $s$ is an integer (see, e.g. \cite{Duo}, p. 163). We will show that
\begin{multline}\label{lemma}
\sup_{k>0,x\in\mathbb{R}^n}2^{k(\rho-\frac{n}{2})}\Vert m(x,\cdot)\psi(2^{-k}|\cdot|)\Vert_{H^\rho}=\sup_{k>0,x\in\mathbb{R}^n}\Vert m(x,2^{k}\cdot)\psi(|\cdot|)\Vert_{H^\rho}<\infty,
\end{multline}
provided that $\rho$ is an integer. From the estimate
\begin{equation}
\Vert m(x,2^{k}\cdot)\psi(|\cdot|)\Vert_{H^\rho}\asymp \Vert m(x,2^{k}\cdot)\psi(|\cdot|)\Vert'_{H^\rho}=\sum_{|\beta|\leq s}\Vert \partial_{\xi}^\beta( m(x,2^{k}\cdot)\psi(|\cdot|) )\Vert_{L^2(\mathbb{R}^n)},
\end{equation}
we will estimate the $L^2$-norms of the derivatives $\partial_{\xi}^\beta( m(x,2^{k}\cdot)\psi(|\cdot|) )(\xi).$ Because the function $\psi$ is supported in some closed interval not containing the origin, the function $\psi(|\cdot|) $ is smooth.  By the  Leibniz rule we have
 $$ \partial_{\xi}^\beta( m(x,2^{k}\xi)\psi(|\xi|) )=\sum_{|\alpha|\leq |\beta|}2^{k|\alpha|}(\partial_{\xi}^\alpha m)(x,2^{k}\xi)\partial_\xi^{\beta-\alpha}\psi(|\xi|).$$
So, we obtain
\begin{equation}\label{dilatationestimate}
\Vert \partial_{\xi}^\beta( m(x,2^{k}\cdot)\psi(|\cdot|) )\Vert_{L^2}\leq  \sum_{|\alpha|\leq \rho}C_\alpha \Vert \partial_\xi^{\beta-\alpha}\psi(|\cdot|)\Vert_{L^2},
\end{equation}
where we have used that \eqref{KNcondition} implies the estimate $|2^{k|\alpha|}(\partial_{\xi}^\alpha m)(x,2^{k}\cdot)|\leq C_\alpha,$ for $k$ large enough. Now, \eqref{lemma}  follows by summing both sides of \eqref{dilatationestimate} over $|\beta|\leq \rho.$ We finish the proof by observing that  every $s_{n,p}$ defined in Theorem \ref{TheoremAnnoucenment}, satisfies the upper bound $s_{n,p}\leq 2n,$ and we can obtain the $L^p$-boundedness of $T_m$ by taking $\rho>s_{n,p}$ with $\rho=2n+1.$ A similar analysis shows that $\rho=[3n/2]+1$ implies the $L^2-$boundedness of $T_m.$
\end{proof}

\begin{corollary}\label{DiscTangCond} Let us consider a complex-valued function $m$ on $\mathbb{R}^n\times\mathbb{Z}^n$ and a pseudo-multiplier $T_m$ with symbol $\{m(x,\nu)\}_{x\in\mathbb{R}^n,\nu\in\mathbb{N}_0^n}.$ If $m$ satisfies the discrete difference conditions
\begin{equation}
|\Delta_\nu^{\alpha} m(x,\nu)|\leq C_{\alpha}(1+|\nu|)^{-|\alpha|},\,\,\alpha\in\mathbb{N}^n_0,\,|\alpha|\leq \rho,
\end{equation} for $\rho=[3n/2]+1,$ then $T_m$ extends to a bounded operator on $L^2(\mathbb{R}^n)$. Moreover, for $\rho=2n+1$ we have the $L^p(\mathbb{R}^n)$-boundedness of $T_m$ for all $1<p<\infty.$
\end{corollary}  
\begin{proof} Let us define for every $z_0\in\mathbb{R}^n,$ the function $m_{z_0}$ given by $m_{z_0}(\nu)=m(z_0,\nu).$
Then we have the estimates
\begin{equation}
|\Delta_\nu^{\alpha} m_{z_0}(\nu)|\leq C_{\alpha}(1+|\nu|)^{-|\alpha|},\,\,\alpha\in\mathbb{N}^n_0,\,|\alpha|\leq \rho.
\end{equation}
From Corollary 4.5.7 of \cite{Ruz}, there exists a suitable function $\tilde{m}_{z_0}$ defined on $\mathbb{R}^n$ such that $\tilde{m}_{z_0}|_{\mathbb{Z}^n}={m}_{z_0}$ and additionally satisfying the conditions,
\begin{equation}
|\partial_\xi^{\alpha} \tilde{m}_{z_0}(\xi)|\leq C_{\alpha}(1+|\xi|)^{-|\alpha|},\,\,\alpha\in\mathbb{N}^n_0,\,|\alpha|\leq \rho.
\end{equation} The function $\tilde{m}$ defined by $\tilde{m}(z_0,\xi):=\tilde{m}_{z_0}(\xi)$ satisfies \eqref{KNcondition}, and by Proposition \ref{KNcondition'} we obtain the $L^2$-boundedness of the pseudo-multiplier $T_{ \tilde{m}  }$ with symbol  $\{\tilde{m}(x,\nu)\}_{x\in\mathbb{R}^n,\nu\in\mathbb{N}_0^n},$ if $\rho=[3n/2]+1$ (or the $L^p$-boundedness, for all $1<p<\infty$ if $\rho=2n+1$). We finish the proof by observing that   $T_{ \tilde{m}  }= T_{ m  } $ in view of the equality sets $\{\tilde{m}(x,\nu)\}_{x\in\mathbb{R}^n,\nu\in\mathbb{N}_0^n}=\{{m}(x,\nu)\}_{x\in\mathbb{R}^n,\nu\in\mathbb{N}_0^n}.$ 
\end{proof}

\subsection{$(L^p,L^q)$-boundedness of spectral  pseudo-multipliers } Let us assume $n\in\mathbb{N},$ arbitrary but fixed.  Let us define the set $2\mathbb{N}_0+n:=\{2m+n:m\in\mathbb{N}_0\}.$ We will consider continuous functions $m(x,\xi)$ defined on $\mathbb{R}^{n}_{x}\times \mathbb{R}_{\xi}$ and we will denote by $m(x,\ell)$  the restriction of $m(x,\xi)$ to the set $\mathbb{R}^n\times (2\mathbb{N}_0+n),$  so that $x\in\mathbb{R}^n$ and $\ell\in 2\mathbb{N}_0+n.$
If we set $\mathscr{F}:L^2(-\infty,\infty)\rightarrow L^2(-\infty,\infty)$ for the one-dimensional Fourier transform, we will consider  symbols $m(x,\ell):=m(x,\xi)|_{\mathbb{R}^n\times (2\mathbb{N}_0+n)}$  satisfying  the, so called, H\"ormander condition of order $s>0,$
\begin{equation}\label{h}
\Vert m\Vert_{l.u.,h^s}:=\sup_{x\in\mathbb{R}^n,k>0}2^{k(s-\frac{n}{2})}\left(\int\limits_{-\infty}^{\infty} \langle t\rangle^{2s}|\mathscr{F}[m(x,\cdot)\psi(2^{-k}|\cdot|)](t)|^2dt \right)^{\frac{1}{2}}<\infty,
\end{equation}
where the function $\psi\in \mathscr{D}(0,\infty)$ satisfies $\psi(t)=1$  for all  $t\in [1,2].$ With the previous notation we want to investigate the H\"ormander condition for pseudo-multipliers of the  form
\begin{equation}
m(x,H)f(x):=\sum_{\ell=0}^{\infty}m(x,\ell)P_{\ell}f(x),
\end{equation}
where we have  denoted by $P_{\ell}$  the orthogonal projection to the subspace generated by the set $\{\phi_\nu:|\nu|=\ell\}.$ For symbols $m(x,\ell)=m(\ell)$ depending only on the $\ell$ variables we have used in \eqref{radial} the term {\textit{radial symbols}}.  If it depends on $x$ we can talk about them as spectral pseudo-multipliers. In the next theorem, we prove that for symbols satisfying the H\"ormander condition of order $s,$ for $s$ suitable, the corresponding spectral pseudo-multipliers are bounded operators from $L^p(\mathbb{R}^n)$ into $L^q(\mathbb{R}^n)$ when $\frac{1}{q}+\frac{1}{p}=1.$ We will denote $\delta(p):=n|\frac{1}{p}-\frac{1}{2}|-\frac{1}{2}$ and $q=p'.$

\begin{theorem}\label{LpLp'}  Let us consider a function $m$  satisfying  \eqref{h}.
Let $m(x,H)$ be a spectral pseudo-multiplier with symbol $\{m(x,\ell)\}_{x\in\mathbb{R}^n,\ell\in 2\mathbb{N}_0+n}.$ Under one of the following conditions
\begin{itemize}
\item $n\geq 2,$ $1\leq p\leq \frac{2n}{n+2} $  and $s>\frac{n+1}{2}+\delta(p),$
\item $n\geq 2,$ $ \frac{2n}{n+2}< p\leq 2 $ and $s>\frac{3n}{2},$
\item $n=1,$ $\frac{4}{3}<p\leq 2,$ $s>2-\frac{1}{p},$
\item $n=1,$ $p=\frac{4}{3},$ $s>\frac{3}{2},$
\item $n=1,$ $1<p<\frac{4}{3},$ $s>1+\frac{1}{3p},$
\end{itemize} the operator $m(x,H)$ extends to a bounded operator  from $L^{p}(\mathbb{R}^n)$ into $L^{p'}(\mathbb{R}^n).$
\end{theorem}
\begin{proof}
In order to prove  Theorem \ref{LpLp'} we will split the symbol $m$ as
 \begin{equation}
  m(x,\ell)=\sum_{k=0}^{\infty}  m_k(x,\ell),\,\,\,\,\, m_k(x,\ell):= m(x,\ell)\cdot 1_{\{2^{k}\leq \ell<2^{k+1}\}}.
 \end{equation}
Let us denote by $T_{m(j)}  $ the pseudo-multiplier associated to $m_j,$  for $j\geq 0.$ Then the operator series
\begin{equation}
T_0+\sum_{k} T_{m(k)}
\end{equation}
converges to $T_m$ in the strong topology on $\mathscr{B}(L^{p}(\mathbb{R}^n),L^{p'}(\mathbb{R}^n))$ and
\begin{equation}
\Vert m(x,H) \Vert_{\mathscr{B}(L^{p}(\mathbb{R}^n),L^{p'}(\mathbb{R}^n))}  \leq \sum_k  \Vert T_{m(k)} \Vert_{\mathscr{B}(L^{p}(\mathbb{R}^n),L^{p'}(\mathbb{R}^n))} .
\end{equation}

So, we want to estimate every norm $\Vert T_{m(j)} \Vert_{\mathscr{B}(L^{p}(\mathbb{R}^n),L^{p'}(\mathbb{R}^n))}.$
For this, we will use the fact that for $f\in C^\infty_{0}(\mathbb{R}^n),$
\begin{equation}
\Vert T_{m(j)}f \Vert_{L^{p'}(\mathbb{R}^n)}=\sup\{ ( T_{m(j)} f,g)_{L^2(\mathbb{R}^n)}\, :\, \Vert g\Vert_{L^{p}(\mathbb{R}^n)}=1 \}.
\end{equation}
In fact, for $f$ and $g$ as above we have
\begin{align*}
( T_{m(k)} f,\overline{g})_{L^2(\mathbb{R}^n)} &=\int_{\mathbb{R}^n}T_{m(k)}f(x)g(x)dx\\
&=\int_{\mathbb{R}^n}\sum_{2^k\leq\ell<2^{k+1}}m(x,\ell)P_{\ell}f(x)g(x)dx\\
&=\sum_{2^k\leq\ell<2^{k+1}}\int_{\mathbb{R}^n}m(x,\ell)P_{\ell}f(x)g(x)dx.
\end{align*}

For every $x\in\mathbb{R}^n$ let us denote the one-dimensional Fourier transform of  $m(x,\cdot)\psi(2^{-k}|\cdot|)$ by $\mathscr{F}[m(x,\cdot)\psi(2^{-k}|\cdot|)].$ So, we have
\begin{equation}
m_k(x,\ell)=\mathscr{F}^{-1}(\mathscr{F}[m(x,\cdot)\psi(2^{-k}|\cdot|)])(\ell)=\int\limits_{-\infty}^{\infty} \mathscr{F}[m(x,\cdot)\psi(2^{-k}|\cdot|)](t)e^{i2\pi \ell\cdot t}dt.
\end{equation}
Consequently,  we have
\begin{align*}
&|( T_{m(k)} f,\overline{g})_{L^2(\mathbb{R}^n)}|\\
&\leq \sum_{2^k\leq \ell<2^{k+1}}\sup_{x\in\mathbb{R}^n}\int\limits_{-\infty}^{\infty} |\mathscr{F}[m(x,\cdot)\psi(2^{-k}|\cdot|)](t)|dt \\
&\hspace{8cm}\times \Vert P_\ell f\Vert_{L^{p'}}\Vert g\Vert_{L^{p}}.
\end{align*} Now, let us fix $n\in\mathbb{N},$ $n\geq 2.$ By taking into account the  Karadzhov's estimate (see Thangavelu \cite{Thangavelu2}, p. 268)
\begin{equation}
\Vert P_{\ell}f\Vert_{L^{p'}(\mathbb{R}^n)}\leq C_p \ell^{\delta(p)-\frac{1}{2}}\Vert f\Vert_{L^p},\,\,\,1\leq p\leq\frac{2n}{n+2},\,\,\,n\geq 2,
\end{equation}
we can estimate the operator norm of $T_{m(k)}$ by
\begin{align*}
&\Vert T_{m(k)} \Vert_{\mathscr{B}(L^p,L^{p'})}\\
&\lesssim \sum_{2^k\leq\ell<2^{k+1}}\sup_{x\in\mathbb{R}^n}\int\limits_{-\infty}^{\infty} |\mathscr{F}[m(x,\cdot)\psi(2^{-k}|\cdot|)](t)|dt \ell^{\delta(p)-\frac{1}{2}}\\
&\lesssim \sum_{2^k\leq\ell<2^{k+1}}\sup_{x\in\mathbb{R}^n}\left(\int\limits_{-\infty}^{\infty}\langle t\rangle^{2s}|\mathscr{F}[m(x,\cdot)\psi(2^{-k}|\cdot|)](t)|^{2} dt\right)^{\frac{1}{2}}\Vert \langle \cdot\rangle^{-s}\Vert_{L^2(-\infty,\infty)}\ell^{\delta(p)-\frac{1}{2}},
\end{align*}
if we impose $s>\frac{1}{2}.$ If additionally we consider the hypothesis \eqref{h}, that is,
\begin{equation}
\sup_{x\in\mathbb{R}^n}\left(\int\limits_{-\infty}^{\infty}\langle t\rangle^{2s}|\mathscr{F}[m(x,\cdot)\psi(2^{-k}|\cdot|)](t)|^{2} dt\right)^{\frac{1}{2}}\leq \Vert m\Vert_{l.u.,{h}^s}\cdot 2^{-k(s-\frac{n}{2})},
\end{equation}
then we have
\begin{align*}
\Vert T_{m(k)} \Vert_{\mathscr{B}(L^p,L^{p'})} &\lesssim \sum_{2^k\leq\ell<2^{k+1}}\Vert m\Vert_{l.u.,{h}^s}\cdot 2^{-k(s-\frac{n}{2})+k(\delta(p)-\frac{1}{2})}\\
&\asymp 2^{k-k(s-\frac{n}{2})+k(\delta(p)-\frac{1}{2})}\Vert m\Vert_{l.u.,{h}^s}=2^{-k(s-\frac{n+1}{2}-\delta(p))}\Vert m\Vert_{l.u.,{h}^s}.
\end{align*}

As a consequence we obtain  $$0\leq \sum_{k=1}^{\infty} \Vert T_{m(k)} \Vert_{\mathscr{B}(L^p,L^{p'})}   \lesssim \sum_{k=1}^\infty2^{-k(s-\frac{n+1}{2}-\delta(p))} \Vert m\Vert_{l.u., {h}^s}<\infty,$$ for $s>\frac{n+1}{2}+\delta(p).$ So, we have
$$\Vert m(x,H) \Vert_{\mathscr{B}(L^p,L^{p'})}\leq  C \Vert m \Vert_{l.u.,{h}^s},\,\,\,1\leq p\leq\frac{2n}{n+2}. $$
Let us note that $\delta(\frac{2n}{n+2})=\frac{1}{2},$ and consequently $T_m:L^{\frac{2n}{n+2}}\rightarrow L^{\frac{2n}{n-2}}$ extends to a bounded operator provided that $s>\frac{n+2}{2}.$ If we assume $s>\frac{3n}{2}$ then from Theorem \ref{TheoremAnnoucenment} we obtain the $L^2$-boundedness of $m(x,H).$ Thus, by the real interpolation we obtain the boundedness of $m(x,H)$ from $L^p$ into $L^{p'}$ for $\frac{2n}{n+2}\leq p\leq 2$ provided that $s>\frac{3n}{2}.$

Now, if $n=1,$ by the Cauchy-Schwarz  inequality we have  for every $\ell\in\mathbb{N}_0$ the estimate,
\begin{equation}\label{estimateonedimensional}
\Vert P_\ell f \Vert_{L^{p'}(\mathbb{R})}=\Vert (f,\phi_{\ell})_{L^2(\mathbb{R})} \phi_{\ell} \Vert_{L^{p'}(\mathbb{R})}\leq \Vert \phi_\ell \Vert_{L^{p'}(\mathbb{R})}^2\Vert f\Vert_{L^p(\mathbb{R})}.
\end{equation}
By Thangavelu's Lemma \ref{Lemma11}, we can estimate the $L^{p'}(\mathbb{R})$-norm of the function $P_{\ell}f$ as follows:
\begin{equation}\label{normPrj1}
  \Vert P_\ell f \Vert_{L^{p'}(\mathbb{R})}\lesssim \ell^{2(\frac{1}{2p'}-\frac{1}{4})}\Vert f \Vert_{L^p(\mathbb{R})}=\ell^{\frac{1}{2}-\frac{1}{p}} \Vert f \Vert_{L^p(\mathbb{R})} ,\,\,\,   \frac{4}{3}<p\leq 2, \end{equation}
\begin{equation}\label{normPrj2}
  \Vert P_\ell f \Vert_{L^{p'}(\mathbb{R})}\lesssim \ell^{2(-\frac{1}{6p'}-\frac{1}{12})}\Vert f \Vert_{L^p(\mathbb{R})}=\ell^{-\frac{1}{2}+\frac{1}{3p}} \Vert f \Vert_{L^p(\mathbb{R})} ,\,\,\,   1<p<\frac{4}{3}.
\end{equation}   
    
    Recalling that for all  $n\in\mathbb{N}$  we have,
\begin{align*}
&|( T_{m(k)} f,\overline{g})_{L^2(\mathbb{R}^n)}|\\
&\leq \sum_{2^k\leq \ell<2^{k+1}}\sup_{x\in\mathbb{R}^n}\int\limits_{-\infty}^{\infty} |\mathscr{F}[m(x,\cdot)\psi(2^{-k}|\cdot|)](t)|dt \\
&\hspace{8cm}\times \Vert P_\ell f\Vert_{L^{p'}(\mathbb{R}^n)}\Vert g\Vert_{L^{p}(\mathbb{R}^n)},
\end{align*}
 we obtain, with $n=1,$
 \begin{align*}
   \Vert  T_{m(k)} f \Vert_{L^{p'}(\mathbb{R}^n)} &\leq    \sum_{2^k\leq \ell<2^{k+1}}2^{-k(s-\frac{n}{2})}\Vert m\Vert_{l.u.,h^s}\Vert P_{\ell}f\Vert_{L^{p'}(\mathbb{R}^n)}\\
    &\lesssim    2^{-k(-1+s-\frac{n}{2})}\Vert m\Vert_{l.u.,h^s}\Vert P_{\ell}f\Vert_{L^{p'}(\mathbb{R}^n)}.
 \end{align*}
 Thus, by \eqref{normPrj1} and  the  estimate above we have,
\begin{align*}0\leq \Vert m(x,H)\Vert_{ \mathscr{B}(L^p,L^{p'})  } &\leq \sum_{k=1}^{\infty} \Vert T_{m(k)} \Vert_{\mathscr{B}(L^p,L^{p'})}   \lesssim \sum_{k=1}^\infty2^{-k(-1+s-\frac{n}{2}-\frac{1}{2}+\frac{1}{p})} \Vert m\Vert_{l.u., {h}^s}\\
&=\sum_{k=1}^\infty2^{-k(s-2+\frac{1}{p})} \Vert m\Vert_{l.u., {h}^s} <\infty,\end{align*} for $s>2-\frac{1}{p},$ when $\frac{4}{3}<p\leq 2,$ and
\begin{align*}0\leq \Vert m(x,H)\Vert_{ \mathscr{B}(L^p,L^{p'})  } &\leq \sum_{k=1}^{\infty} \Vert T_{m(k)} \Vert_{\mathscr{B}(L^p,L^{p'})}   \lesssim \sum_{k=1}^\infty2^{-k(-1+s-\frac{n}{2}+\frac{1}{2}-\frac{1}{3p})} \Vert m\Vert_{l.u., {h}^s}\\
&=  \sum_{k=1}^\infty2^{-k(s-1-\frac{1}{3p})} \Vert m\Vert_{l.u., {h}^s}   <\infty,\end{align*} for $s>1+\frac{1}{3p},$ and $1<p<\frac{4}{3},$ in view of \eqref{normPrj2}.  The $(L^{\frac{4}{3}},L^{4})$-boundedness of $m(x,H)$ now follows by the real interpolation for $s>\frac{3}{2}.$  In fact, if we fix $s>\frac{3}{2}$ there exists $p_{1}>\frac{4}{3}$ satisfying   $s>2-\frac{1}{p_1}>\frac{3}{2}.$ We also have the existence of $p_0>0,$ $1<p_0<\frac{4}{3}$ such that $s>\frac{3}{2}>1+\frac{1}{3p_0}.$
Thus, $m(x,H)$ admits bounded extensions from $L^{p_{0}}(\mathbb{R}^n)$ into $L^{p_{0}'}(\mathbb{R}^n)$ and from $L^{p_1}(\mathbb{R}^n)$ into $L^{p_1'}(\mathbb{R}^n)$ respectively. By the inequality $p_0<\frac{4}{3}<p_1$ and the real interpolation we deduce that $m(x,H)$ has a bounded extension from  $L^{\frac{4}{3}}(\mathbb{R}^n)$ into $L^{4}(\mathbb{R}^n).$ Thus, we have completed  the proof.
\end{proof}

Now, by the real interpolation we give the following general $(L^p,L^q)$ boundedness theorem.

\begin{theorem}\label{Teoremita:Lp:Lq;general}
Let us consider a function $m=m(x,\ell)$  satisfying  \eqref{h}.
Let $m(x,H)$ be a spectral pseudo-multiplier with symbol $\{m(x,\ell)\}_{x\in\mathbb{R}^n,\ell\in 2\mathbb{N}_0+n}.$ Under one of the following conditions
\begin{itemize}
\item  $n\geq 2,$  $1< p\leq \frac{2n}{n+2} $  and $s>\frac{3n-1}{2}+n(\frac{1}{2}-\frac{1}{p}),$
\item $n\geq 2,$  $ \frac{2n}{n+2}< p\leq 2 $ and $s>\frac{3n}{2},$
\item $n=1,$ $\frac{4}{3}\leq p<2,$ $s>\frac{3}{2},$
\item $n=1,$ $1<p<\frac{4}{3},$  $s>1+\frac{1}{3p},$
\end{itemize} the operator $m(x,H)$ extends to a bounded operator  from $L^{p}(\mathbb{R}^n)$ into $L^{q}(\mathbb{R}^n),$ for all $p\leq q\leq p'.$
\end{theorem}
\begin{proof} Let us observe that $\delta(\frac{2n}{n+2})=\frac{1}{2}\leq \delta(p)\leq \frac{n-1}{2}$ for $1\leq p\leq \frac{2n}{n+2}$ and that $\frac{1}{2}-\frac{1}{p}\leq 0$ for $1<p\leq 2.$ With these inequalities in mind,
from Theorems \ref{TheoremAnnoucenment} and \ref{LpLp'}, under one of the following conditions
\begin{itemize}
\item $n\geq 2,$ $1<p\leq \frac{2n}{n+2},$ $$ s>\frac{3n-1}{2}+n(\frac{1}{2}-\frac{1}{p})=\max \{\frac{n+1}{2}+\delta(p),\frac{3n-1}{2}+n(\frac{1}{2}-\frac{1}{p})\},$$
\item $n\geq 2,$ $\frac{2(n+3)}{n+5}\leq  p\leq 2,$ $$s>\frac{3n}{2}=\max\{ \frac{3n}{2}, \frac{3n}{2}+\frac{n-1}{2}(\frac{1}{2}-\frac{1}{p})\},$$
\item $n\geq 2,$ $\frac{2n}{n+2}\leq p\leq \frac{2(n+3)}{n+5},$ $$s>\frac{3n}{2}=\max\{ \frac{3n}{2}, \frac{3n}{2}-\frac{1}{6}+\frac{2n}{3}(\frac{1}{2}-\frac{1}{p})\},$$
\item $n=1,$ $\frac{4}{3}\leq p<2,$ $$s>\frac{3}{2}=\max\{ \frac{3}{2},2-\frac{1}{p} \},$$
\item $n=1,$ $1<p<\frac{4}{3},$  $$s>1+\frac{1}{3p}=\max\{ 1+\frac{1}{3p},\frac{4}{3} +\frac{2}{3}(\frac{1}{2}-\frac{1}{p}) \},$$
\end{itemize} the spectral pseudo-multiplier $m(x,H)$ extends to a bounded operator from $L^p(\mathbb{R}^n)$ into $L^{p'}(\mathbb{R}^n)$ and also we have its boundedness from $L^p(\mathbb{R}^n)$ into $L^{p}(\mathbb{R}^n).$ So, by the Riesz-Thorin interpolation theorem we deduce the boundedness of $m(x,H)$ from $L^p(\mathbb{R}^n)$ into $L^{q}(\mathbb{R}^n)$ for all $p\leq q\leq p'.$ So, we finish the proof.
\end{proof}

\subsection{Lower bounds for the operator norm of multipliers on $L^p$ spaces} Now, we estimate from below the operator norm of multipliers associated to the harmonic oscillator.

\begin{theorem}\label{lowerboundnorms} Let $1\leq p\leq \infty.$ 
Let us assume that $T_m$ is a multiplier associated to the harmonic oscillator. If $T_m$ is a bounded operator on $L^p(\mathbb{R}^n),$ then we have the following lower bound for the $L^p$-operator norm of $T_m,$

$$\Vert T_m\Vert_{\mathscr{B}(L^p(\mathbb{R}^n))}\geq \sup_{\nu\in\mathbb{N}^n}|m(\nu)| .$$

\end{theorem}
\begin{proof}
For the proof we can take advantage of the orthogonality properties of the Hermite functions $\phi_\nu,$ $\nu\in\mathbb{N}^n_0.$ By definition we have
\begin{equation}\label{lowerbound}
T_{m}(\phi_\nu)=m(\nu)\phi_\nu.
\end{equation}
As consequence we obtain
\begin{align*}
\Vert T_m\Vert_{\mathscr{B}(L^p(\mathbb{R}^n))}\geq\Vert T_{m}(\frac{\phi_\nu}{\Vert \phi_\nu\Vert_{L^p}}) \Vert_{L^p(\mathbb{R}^n)}=|m(\nu)|.
\end{align*}
Thus, we end the proof.
\end{proof}

\section{Compactness of pseudo-multipliers}\label{compactnesssection}

\subsection{$L^2$-compactness of multipliers}
Now, we use the Fourier analysis produced by the harmonic oscillator in order to characterise the $L^2$-compactness of multipliers. The following is an analogue of a criterion very well known in other settings.

\begin{theorem}\label{CompL2}
Let us assume that  $T_m$ is a bounded multiplier   on $L^2(\mathbb{R}^n).$ Then, $T_m$ is a compact operator on $L^2(\mathbb{R}^n)$ if and only if
$
\lim_{|\nu|\rightarrow \infty}m(\nu)=0.
$
\end{theorem}
\begin{proof}
In order to prove the theorem, let us first assume that $T_m$ is an $L^2$-compact operator. If $f\in L^2(\mathbb{R}^n),$ by the Plancherel theorem we have
\begin{equation}
\Vert f \Vert_{L^2(\mathbb{R}^n)}^2=\sum_{\nu\in\mathbb{N}^n_0}|( f,\phi_\nu )_{L^2}|^2.
\end{equation}
Consequently, we have $( f,\phi_\nu )\rightarrow 0$ as $|\nu|\rightarrow \infty.$ So, we conclude that in $L^2(\mathbb{R}^n),$ the sequence $\{\phi_\nu\}_{\nu\in\mathbb{N}^n_0}$ converges weakly to zero. By the compactness of $T_m$ the sequence $\{T(\phi_\nu)\}_{\nu\in\mathbb{N}^n_0}$ converges to zero in the $L^2$-norm. So,
\begin{equation}
\lim_{|\nu|\rightarrow \infty}\Vert T_m \phi_\nu \Vert_{L^2(\mathbb{R}^n)}=\lim_{|\nu|\rightarrow \infty}|m(\nu)|=0.
\end{equation} For the proof of the converse assertion, let us assume that the sequence $\{m(\nu)\}_{\nu\in\mathbb{N}^n_0}$ tends to zero as $|\nu|\rightarrow \infty.$ In order to show that $T_m$ is compact, we will approximate it with operators of finite rank. So, let us define the sequence of finite rank operators $T_{m(k)},$ $k\in\mathbb{N},$ by
\begin{equation}
T_{m(k)}f:=\sum_{|\nu|\leq k}m(\nu)\widehat{f}(\phi_\nu)\phi_\nu.
\end{equation} By the orthogonality of the Hermite functions, we have
\begin{equation}
\Vert T_{m(k)}f-T_m f\Vert_{L^2(\mathbb{R}^n)}^2=\sum_{|\nu|\geq k}|m(\nu)|^2|( f,\phi_\nu )_{L^2}|^2\leq \sup_{|\nu|\geq k}|m(\nu)|^2\Vert f\Vert_{L^2(\mathbb{R}^n)}^2.
\end{equation} So, we obtain
\begin{equation}
\lim_{k\rightarrow\infty}\Vert T_{m(k)}-T_m \Vert_{\mathscr{B}(L^2)}\leq \lim_{k\rightarrow\infty} \sup_{|\nu|\geq k}|m(\nu)|=0.
\end{equation} With the last line we finish the proof.
\end{proof}

\subsection{$L^p$-compactness and $L^p$-boundedness for multipliers via Littlewood-Paley theory}

In the preceding subsection we have characterised the compactness on $L^2(\mathbb{R}^n)$ of multipliers with the Plancherel theorem as a fundamental tool. In order to investigate the $L^p$-compactness of multipliers for $1<p<\infty,$ but $p\neq2,$ we will use the Littlewood-Paley theorem (which is a partial substitute of the Plancherel theorem on $L^p$-spaces) associated to dyadic decompositions of the spectrum of the harmonic oscillator. The main notion in the Littlewood-Paley theory is the concept of a dyadic decomposition.  Here, the sequence $\{\psi_{l}\}_{l\in\mathbb{N}_{0}}$ is a dyadic decomposition,  defined as follows:  we choose a function $\psi\in \mathscr{D}(0,\infty)$  supported
in  $[1/2, 1],$ $\psi=1$ on $[2/3,4/5].$ Denote by $\psi_{l}$ the function $\psi_{l}(t)=\psi(2^{-l}t),$ $t\in \mathbb{R}.$ For some smooth compactly supported  function $\psi_{0}$ we have
\begin{eqnarray}\label{deco1}
\sum_{l\in\mathbb{N}_{0}}\psi_{l}(\lambda)=1,\,\,\, \text{for every}\,\,\, \lambda>0.
\end{eqnarray} Now we present the Littlewood-Paley Theorem in the form of the following result (see Theorems 7.1 and 7.3 of  \cite{BuiDuong} and Proposition 5 of \cite{PeXu}).

\begin{theorem}\label{LPT}
Let  $1<p<\infty$ and for every $l\in\mathbb{N}_0,$ let us consider the multipliers $T_{\psi_l}$ given by
\begin{equation}
T_{\psi_l}f(x):=\sum_{2^l\leq \langle\nu\rangle<2^{l+1}}\psi_l(\langle\nu\rangle)\phi_{\nu}(x)\widehat{f}(\phi_\nu),\,\,\langle\nu \rangle:=(1+|\nu|^2)^{\frac{1}{2}}.
\end{equation}
Then there exist  constants $0<c_p,C_{p}<\infty$ depending only on $p$  such that
\begin{equation}\label{LPTequ}
c_p\Vert f\Vert_{L^{p}(\mathbb{R}^n)}\leq \left\Vert \left(\sum_{l=0}^{\infty} |T_{\psi_l}f(x)|^{2}    \right)^{\frac{1}{2}}\right\Vert_{L^{p}(\mathbb{R}^n)}\leq  C_{p}\Vert f\Vert_{L^{p}(\mathbb{R}^n)},
\end{equation}
holds for all $f\in L^{p}(\mathbb{R}^n).$
\end{theorem}
The following $L^p$ multiplier theorem provides sufficient conditions for the $L^p$-boundedness of multipliers (different from the H\"ormander-Mihlin condition) and their $L^p$-compactness.
\begin{theorem}\label{LpTLpComp}
Let us assume that $T_m$ is a multiplier and let $1<p<\infty.$ Let us assume  that there exists a sequence $\{\nu_l\}$ satisfying: $2^l\leq |\nu_l|<2^{l+1},$ $m(\nu_l)\neq 0$ for every $l\in\mathbb{N}_0$ and
\begin{equation}\label{asymptoticproperty}
\lim_{l\rightarrow\infty}\frac{m(\nu_l')}{m(\nu_l)}=K\neq 0,
\end{equation} for every sequence $\{\nu_l'\}$ where $2^l\leq |\nu_l'|<2^{l+1}$ (the constant $K$ depends on the sequences $\nu_l$ and $\nu_l'$). Then,
\begin{itemize}
\item if $\Vert m\Vert_{L^\infty(\mathbb{N}_0^n)}<\infty,$ then the operator $T_m$ extends to a bounded operator on $L^p(\mathbb{R}^n)$ and \begin{equation}
\Vert T_m\Vert_{\mathscr{B}(L^p(\mathbb{R}^n))}\leq C \Vert m\Vert_{L^\infty(\mathbb{N}_0^n)}.
\end{equation}
\item if $|m(\nu)|\rightarrow 0$ as $|\nu|\rightarrow \infty,$ then the operator $T_m$ extends to a compact operator on $L^p(\mathbb{R}^n).$
\end{itemize}
\end{theorem}
\begin{proof}
Let us assume that $f\in L^p(\mathbb{R}^n),$ $1<p<\infty.$ By the Littlewood-Paley Theorem (see Theorem \ref{LPT} above) we have
\begin{align*}
&\Vert T_m f\Vert_{L^{p}(\mathbb{R}^n)} \lesssim  \left\Vert \left(\sum_{l=0}^{\infty} |T_{\psi_l}T_{m}f(x)|^{2}    \right)^{\frac{1}{2}}\right\Vert_{L^{p}(\mathbb{R}^n)}
\\&= \left\Vert \left(\sum_{l=0}^{\infty}  |  \sum_{2^l\leq \langle\nu\rangle<2^{l+1}}\psi_l(\langle\nu\rangle)m(\nu)\phi_{\nu}(x)\widehat{f}(\phi_\nu)   |^2  \right)^{\frac{1}{2}}\right\Vert_{L^{p}(\mathbb{R}^n)}.
\end{align*}
Taking into account both, that $m$ is bounded and the condition \eqref{asymptoticproperty} we have
\begin{align*}
& \left\Vert  \left(\sum_{l=0}^{\infty}  |  \sum_{2^l\leq \langle\nu\rangle<2^{l+1}}\psi_l(\langle\nu\rangle)m(\nu)\phi_{\nu}(x)\widehat{f}(\phi_\nu)   |^2  \right)^{\frac{1}{2}}\right\Vert_{L^{p}(\mathbb{R}^n)}\\
& \hspace{2cm}= \left\Vert  \left(\sum_{l=0}^{\infty} |m(\nu_l) |^2|  \sum_{2^l\leq \langle\nu\rangle<2^{l+1}}\psi_l(\langle\nu\rangle)\frac{m(\nu)}{m(\nu_l)}\phi_{\nu}(x)\widehat{f}(\phi_\nu)   |^2  \right)^{\frac{1}{2}}\right\Vert_{L^{p}(\mathbb{R}^n)}\\
& \hspace{2cm}\leq \Vert m\Vert_{L^\infty(\mathbb{N}_0^n)} \left\Vert  \left(\sum_{l=0}^{\infty} |  \sum_{2^l\leq \langle\nu\rangle<2^{l+1}}\psi_l(\langle\nu\rangle)\frac{m(\nu)}{m(\nu_l)}\phi_{\nu}(x)\widehat{f}(\phi_\nu)   |^2  \right)^{\frac{1}{2}}\right\Vert_{L^{p}(\mathbb{R}^n)}\\
&  \hspace{2cm}\asymp \Vert m\Vert_{L^\infty(\mathbb{N}_0^n)} \left\Vert  \left(\sum_{l=0}^{\infty} |  \sum_{2^l\leq \langle\nu\rangle<2^{l+1}}\psi_l(\langle\nu\rangle)\phi_{\nu}(x)\widehat{f}(\phi_\nu)   |^2  \right)^{\frac{1}{2}}\right\Vert_{L^{p}(\mathbb{R}^n)}\\
&  \hspace{2cm}\lesssim \Vert m\Vert_{L^\infty(\mathbb{N}_0^n)} \left\Vert f \right\Vert_{L^{p}(\mathbb{R}^n)},\\
\end{align*} where in the last line we have used the Littlewood-Paley Theorem \ref{LPT} again. So, we have proved the first part of the theorem. Now, if in addition  $|m(\nu)|\rightarrow 0$ as $|\nu|\rightarrow \infty,$ we will prove that $T_m$ can be approximated by rank finite operators and consequently we obtain the compactness of $T_m.$ Let us define for every $k\in\mathbb{N}$ the operator,
\begin{equation}
T_{m(k)}f:=\sum_{\langle \nu\rangle\leq k}m(\nu)\widehat{f}(\phi_\nu)\phi_\nu.
\end{equation} A similar argument as in the proof of the first assertion shows us that the estimate
\begin{equation}
\Vert T_m f-T_{m(k)}f\Vert_{L^{p}(\mathbb{R}^n)} \lesssim \sup_{\langle \nu\rangle\geq 2^k}| m(\nu)|\Vert f\Vert_{L^p},
\end{equation} holds true. Consequently we have the norm estimates
\begin{eqnarray}
\Vert T_m -T_{m(k)}\Vert_{\mathscr{B}(L^{p}(\mathbb{R}^n))} \lesssim \sup_{\langle \nu\rangle\geq 2^k}| m(\nu)|\rightarrow 0\textnormal{ as }|\nu|\rightarrow\infty.
\end{eqnarray} So, we finish the proof.
\end{proof}
\begin{remark}
 Let us note that $m(\nu):=(1+|\nu|)^{i\tau},$ $\tau\in\mathbb{R},$  satisfies \eqref{asymptoticproperty} and clearly it is a bounded symbol. By the preceding theorem we conclude that $T_{m}$ extends to a bounded operator on $L^p(\mathbb{R}^n),$ $1<p<\infty.$ Also, is easy to see that $m_\kappa(\nu):=(1+|\nu|)^{-\kappa}$ for $\kappa>0$ satisfies \eqref{asymptoticproperty} and $|m_\kappa(\nu)|\rightarrow 0$ as $|\nu|\rightarrow\infty.$ Consequently every operator $T_{m_\kappa}$ extends to a compact operator on $L^p(\mathbb{R}^n)$ for all $1<p<\infty.$
\end{remark}

\section{$L^p$-boundedness for multilinear pseudo-multipliers}\label{multilinear}
In this section we analyse the boundedness of multilinear pseudo-multipliers on Lebesgue spaces which are operators defined by
\begin{equation}\label{pseudo-multiplier'}
T_{m}(f_1,\cdots, f_\varkappa)(x):=\sum_{\nu\in\mathbb{N}_{0}^{n\varkappa}}m(x,\nu)\widehat{f}_1(\phi_{\nu_1})\cdots \widehat{f}_\varkappa(\phi_{\nu_\varkappa}) \phi_{\nu_1}(x)\cdots \phi_{\nu_\varkappa}(x),
\end{equation}
for all $(f_{1},f_{2},\cdots ,f_\varkappa)\in\mathscr{D}(\mathbb{R}^n)^{\varkappa},$ where $\nu:=(\nu_1,\cdots, \nu_\varkappa),$  $\nu_i\in\mathbb{N}_0^n,$ and $x\in\mathbb{R}^{n}.$
In order to prove a general theorem on the boundedness of these operators, we establish the following proposition.
\begin{proposition}\label{firstlemmamultilinear}
Let us consider a pseudo-multiplier $T_m$ defined on $\mathscr{D}(\mathbb{R}^n)^{\varkappa}$ and let $m:\mathbb{R}^n\times \mathbb{N}_0^{n\varkappa}\rightarrow\mathbb{C}$ be its symbol. Let us assume that for $s>0,$ $m$ satisfies the condition
\begin{equation}\label{multilinearhormander}
\Vert m \Vert_{l.u.,\mathcal{H}^s}:=\sup_{k>0,\,x\in\mathbb{R}^n} \,2^{k(s-\frac{n\varkappa}{2})}\Vert  \langle z \rangle^{s} \mathscr{F}_{H}^{-1}[m(x,\cdot)\psi(2^{-k}|\cdot|)](z)\Vert_{L^2({\mathbb{R}}^{n\varkappa}_z)}<\infty,
\end{equation}
and that $\varkappa\geq 2.$ Then
\begin{itemize}
\item[1.] If $s>\frac{3n\varkappa}{2}+{(\varkappa-1)}\gamma_\infty,$ the operator $ T_m$ extends to a bounded multilinear operator from $L^{1}\times L^{\infty}\times\cdots \times L^{\infty}\times L^{\infty}$ into $L^{1}(\mathbb{R}^n),$ and
\begin{equation}
\Vert T_m \Vert_{\mathscr{B}(L^1\times (L^{\infty})^{\varkappa-1},\,L^1)}\leq C(\Vert m \Vert_{l.u.,\mathcal{H}^s}+\Vert m(\cdot,0) \Vert_{L^{\infty}(\mathbb{R}^n)}).
\end{equation}
\item[2.] If $s>\frac{3n\varkappa}{2}+\frac{(\varkappa-1)n}{4}-\frac{1}{12},$ the operator $ T_m$ extends to a bounded multilinear operator from $L^{2}\times L^{\infty}\times\cdots \times L^{\infty}\times L^{\infty}$ into $L^{2}(\mathbb{R}^n),$ and
\begin{equation}
\Vert T_m \Vert_{\mathscr{B}(L^2\times (L^{\infty})^{\varkappa-1},\,L^2)}\leq C(\Vert m \Vert_{l.u.,\mathcal{H}^s}+\Vert m(\cdot,0) \Vert_{L^{\infty}(\mathbb{R}^n)}).
\end{equation}
\item[3.] If  $s>\frac{3n\varkappa}{2}+\frac{(n-1)(\varkappa-1)}{2}+\gamma_p,$ $\gamma_p$ { defined as in } \eqref{gammap}, the operator $ T_m$ extends to a bounded multilinear operator from $L^{p}\times L^{\infty}\times\cdots \times L^{\infty}\times L^{\infty}$ into $L^{p}(\mathbb{R}^n),$ and
\begin{equation}
\Vert T_m \Vert_{\mathscr{B}(L^p\times (L^{\infty})^{\varkappa-1},\,L^p)}\leq C(\Vert m \Vert_{l.u.,\mathcal{H}^s}+\Vert m(\cdot,0) \Vert_{L^{\infty}(\mathbb{R}^n)}),
\end{equation} for all $2< p\leq \infty.$
\end{itemize}
\end{proposition}
\begin{proof}
We proceed with the proof of the first statement. Since
\begin{equation}
\Vert T_m f\Vert_{L^{1}(\mathbb{R}^n)}=\sup_{\Vert g\Vert_{L^{\infty}}=1}|(T_m f,\overline{g})|,
\end{equation}
similar to the previous section we will estimate $|(T_m f,\overline{g})|$ for $\Vert g\Vert_{L^{\infty}}=1.$ Now, for $f:=(f_{1},f_{2},\cdots ,f_\varkappa)\in\mathscr{D}(\mathbb{R}^n)^{\varkappa}$ we have
\begin{equation}
|(T_m f,\overline{g})|\leq |(T_0 f,\overline{g}) |+\sum_{k=0}^{\infty}|(T_{m(k)} f,\overline{g})|,
\end{equation}
where $T_{m(k)}$ is the pseudo-multiplier associated to the symbol $$m_{k}(x,\nu)=m(x,\nu)\cdot 1_{[2^{k},2^{k+1})}(|\nu|),$$ and $T_{0}$ is the operator with symbol $m(x,0)\delta_{\nu,0}.$ For $z_j\in \mathbb{R}^n,$ $z=(z_1,z_2,\cdots ,z_\varkappa)\in \mathbb{R}^{n\varkappa},$ and $\phi_\nu(z)=\phi_{\nu_1}(z_1)\cdots \phi_{\nu_\varkappa}(z_\varkappa) ,$ the inversion formula for the Fourier-Hermite transform gives
\begin{align*}
&|(T_{m(k)} f,\overline{g})|  =\left| \int_{\mathbb{R}^n}    T_{m(k)}f(x)g(x)dx\right| \\
& =| \int_{\mathbb{R}^n} \sum_{ 2^{K}\leq |\nu|<2^{k+1} }m_k(x,\nu)\widehat{f}_1(\phi_{\nu_1})\cdots \widehat{f}_\varkappa(\phi_{\nu_\varkappa}) \phi_{\nu_1}(x)\cdots \phi_{\nu_\varkappa}(x)g(x)dx|\\
&\leq \sum_{ 2^{k}\leq |\nu|<2^{k+1} } | \int_{\mathbb{R}^n} m_k(x,\nu)\widehat{f}_1(\phi_{\nu_1})\cdots \widehat{f}_\varkappa(\phi_{\nu_\varkappa}) \phi_{\nu_1}(x)\cdots \phi_{\nu_\varkappa}(x)g(x)dx|\\
&=  \sum_{ 2^{k}\leq |\nu|<2^{k+1} } | \int_{\mathbb{R}^n} \int_{\mathbb{R}^{n\varkappa}}\phi_{\nu}(z)\mathscr{F}_{H}^{-1}[m_k(x,\cdot)](z)dz\\
&\hspace{4cm}\times\widehat{f}_1(\phi_{\nu_1})\cdots \widehat{f}_\varkappa(\phi_{\nu_\varkappa}) \phi_{\nu_1}(x)\cdots \phi_{\nu_\varkappa}(x)g(x)dx|\\
&\leq  \sum_{ 2^{k}\leq |\nu|<2^{k+1} } \sup_{x\in\mathbb{R}^n} \int_{\mathbb{R}^{n\varkappa}}|\phi_{\nu}(z)||\mathscr{F}_{H}^{-1}[m_k(x,\cdot)]|(z)dz\\
&\hspace{2cm}\times \Vert f_{1}\Vert_{L^{1}}\Vert\phi_{\nu_1}\Vert_{L^{\infty}} \prod_{j=2}^{\varkappa} \Vert f_{j}\Vert_{L^{\infty}}\Vert\phi_{\nu_j}\Vert_{L^{1}}\cdot \Vert g\Vert_{L^\infty}\int_{\mathbb{R}^n}|\phi_{\nu_1}(x)\cdots \phi_{\nu_\varkappa}(x)|dx.\\
\end{align*}
Taking into account that $\varkappa\geq 2$ we write,
\begin{equation}\label{Diferencia1}
\int_{\mathbb{R}^n}|\phi_{\nu_1}(x)\cdots \phi_{\nu_\varkappa}(x)|dx\leq \Vert \phi_{\nu_1} \Vert_{L^2}\Vert \phi_{\nu_2} \Vert_{L^2}\prod_{j\neq 1,2}^{\varkappa}\Vert \phi_{\nu_j} \Vert_{L^{\infty}}\lesssim 1,
\end{equation} where we have used Remark \ref{infty} for the terms in the products.
Consequently,
\begin{align*}
&|(T_{m(k)} f,\overline{g})|  \leq \sum_{2^k\leq |\nu|<2^{k+1}}  \sup_{x\in\mathbb{R}^n} \int_{\mathbb{R}^{n\varkappa}}|\phi_{\nu}(z)||\mathscr{F}_{H}^{-1}[m_k(x,\cdot)]|(z)dz\\
&\hspace{2cm}\times \Vert f_{1}\Vert_{L^{1}}\Vert\phi_{\nu_1}\Vert_{L^{\infty}} \prod_{j=2}^{\varkappa} \Vert f_{j}\Vert_{L^{\infty}}\Vert\phi_{\nu_j}\Vert_{L^{1}}\cdot \Vert g\Vert_{L^\infty}.\\
\end{align*}
Hence we have the following estimate for the norm of $T_{m(k)}$,
\begin{align*}
&\Vert T_{m(k)} \Vert_{\mathscr{B}(L^1\times (L^{\infty})^{\varkappa-1},\,L^1)}\\
&\lesssim  \sum_{ 2^{k}\leq |\nu|<2^{k+1} }\sup_{x\in\mathbb{R}^n} \int_{\mathbb{R}^{n\varkappa}}|\phi_{\nu}(z)||\mathscr{F}_{H}^{-1}[m_k(x,\cdot)]|(z)dz\times \Vert\phi_{\nu_1}\Vert_{L^{\infty}} \prod_{j=2}^{\varkappa} \Vert\phi_{\nu_j}\Vert_{L^{1}}\\
&\leq  \sum_{ 2^{k}\leq |\nu|<2^{k+1} } \sup_{x\in\mathbb{R}^n}\Vert \langle z\rangle^{s}\mathscr{F}_{H}^{-1}[m_k(x,\cdot)]\Vert_{L^2(\mathbb{R}^{n\varkappa})}\\
&\hspace{3cm}\times \Vert \langle z\rangle^{-s} \Vert_{L^2(\mathbb{R}^{n\varkappa})}\Vert \phi_\nu \Vert_{L^{\infty}} \times \Vert\phi_{\nu_1}\Vert_{L^{\infty}} \prod_{j=2}^{\varkappa} \Vert\phi_{\nu_j}\Vert_{L^{1}}.
\end{align*} Since
\begin{equation}
\Vert \phi_\nu \Vert_{L^{\infty}} \times \Vert\phi_{\nu_1}\Vert_{L^{\infty}} \prod_{j=2}^{\varkappa} \Vert\phi_{\nu_j}\Vert_{L^{1}} \leq \Vert\phi_{\nu_1} \Vert_{L^\infty}^2\prod_{j=2}^{\varkappa} \Vert \phi_{\nu_j} \Vert_{L^1}\Vert\phi_{\nu_j}\Vert_{L^\infty}
\end{equation}
and by using the following estimate in Lemma \ref{Lemma1} for $p=\infty,$
\begin{equation}\label{L1estimate}
\Vert  \phi_{\nu_j} \Vert_{L^1(\mathbb{R}^n)}\Vert  \phi_{\nu_j} \Vert_{L^{\infty}(\mathbb{R}^n)}\lesssim |\nu_j|^{\gamma_\infty},\,\,\,2\leq j\leq \varkappa,
\end{equation}
($\gamma_\infty=\frac{n-1}{2}$ for $n\geq 2$ and for $n=1,$  $\gamma_\infty=1/6$), we obtain
\begin{equation}
\Vert \phi_\nu \Vert_{L^{\infty}} \times \Vert\phi_{\nu_1}\Vert_{L^{\infty}} \prod_{j=2}^{\varkappa} \Vert  \phi_{\nu_j} \Vert_{L^1} \lesssim \Vert\phi_{\nu_1}\Vert^2_{L^{\infty}} \prod_{j=2}^{\varkappa}  |\nu_j|^{ \gamma_\infty }\lesssim |\nu|^{{  \gamma_\infty(\varkappa-1)  }}.
\end{equation}
Let us note that in the last estimates we have used that Remark \ref{infty} implies $\Vert \phi_{\nu_1}\Vert_{L^\infty}=O(1)$.  Consequently we have for $s>\frac{n\varkappa}{2},$
\begin{align*}
&\Vert T_{m(k)} \Vert_{\mathscr{B}(L^1\times (L^{\infty})^{\varkappa-1},\,L^1)}\\
&\lesssim  \sum_{ 2^{k}\leq |\nu|<2^{k+1} } \sup_{x\in\mathbb{R}^n}\Vert \langle z\rangle^{s}\mathscr{F}_{H}^{-1}[m_k(x,\cdot)]\Vert_{L^2(\mathbb{R}^{n\varkappa})} \Vert \langle z\rangle^{-s} \Vert_{L^2(\mathbb{R}^{n\varkappa})}|\nu|^{ (\varkappa-1)\gamma_\infty }\\
&\lesssim  \sum_{ 2^{k}\leq |\nu|<2^{k+1} } 2^{-k(s-\frac{n\varkappa}{2})}2^{k(\varkappa-1)\gamma_\infty(\varkappa-1)}\Vert m \Vert_{l.u.\mathcal{H}^s}\asymp 2^{-k(s-\frac{n\varkappa}{2})+k(\varkappa-1)\gamma_\infty+kn\varkappa} \Vert m \Vert_{l.u.\mathcal{H}^s}\\
&=2^{-k(s-\frac{3n\varkappa}{2}-(\varkappa-1)\gamma_\infty)}\Vert m \Vert_{l.u.\mathcal{H}^s}.
\end{align*}
So, we obtain the following upper bound for the series
\begin{align*}
\sum_{k=1}^{\infty} \Vert T_{m(k)}  \Vert_{\mathscr{B}(L^1\times (L^{\infty})^{\varkappa-1},\,L^1)}\leq \Vert m \Vert_{l.u.\mathcal{H}^s}\times \sum_{k=1}^{\infty}2^{-k(s-\frac{3n\varkappa}{2}-(\varkappa-1)\gamma_\infty)}
\end{align*}
which converges provided that $s>\frac{3n\varkappa}{2}+{(\varkappa-1)}\gamma_\infty.$ Now, it is easy to see that $$\Vert T_{0} \Vert_{\mathscr{B}(L^1\times (L^{\infty})^{\varkappa-1},\,L^1)}\lesssim \Vert m(\cdot,0) \Vert_{L^{\infty}(\mathbb{R}^{n})}.$$
As a consequence we get $$  \Vert T_m \Vert_{\mathscr{B}(L^1\times (L^{\infty})^{\varkappa-1},\,L^1)}\leq C(\Vert m \Vert_{l.u.,\mathcal{H}^s}+\Vert m(\cdot,0) \Vert_{L^{\infty}(\mathbb{R}^{n})} ).$$
So, we finish  the proof of the first statement. For the proof of the second statement, we observe that
\begin{align*}
&|(T_{m(k)} f,\overline{g})|\\
&\leq\sum_{ 2^{k}\leq |\nu|<2^{k+1} } \sup_{x\in\mathbb{R}^n} \int_{\mathbb{R}^{n\varkappa}}|\phi_{\nu}(z)||\mathscr{F}_{H}^{-1}[m_k(x,\cdot)]|(z)dz\\
&\hspace{2cm}\times \Vert f_{1}\Vert_{L^{2}}\Vert\phi_{\nu_1}\Vert_{L^{2}} \prod_{j=2}^{\varkappa} \Vert f_{j}\Vert_{L^{\infty}}\Vert\phi_{\nu_j}\Vert_{L^{1}}\cdot \Vert g\Vert_{L^2} \cdot (\int_{\mathbb{R}^n}|\phi_{\nu_1}(x)\cdots \phi_{\nu_\varkappa}(x)|^2dx)^{\frac{1}{2}}.\\
&\lesssim \sum_{ 2^{k}\leq |\nu|<2^{k+1} } \sup_{x\in\mathbb{R}^n} \int_{\mathbb{R}^{n\varkappa}}|\phi_{\nu}(z)||\mathscr{F}_{H}^{-1}[m_k(x,\cdot)]|(z)dz\\
&\hspace{2cm}\times \Vert f_{1}\Vert_{L^{2}}\prod_{j=2}^{\varkappa} \Vert f_{j}\Vert_{L^{\infty}}\Vert\phi_{\nu_j}\Vert_{L^{1}}\cdot \Vert g\Vert_{L^2} |\nu|^{-\frac{1}{12}},
\end{align*}
where we have estimated $(\int_{\mathbb{R}^n}|\phi_{\nu_1}(x)\cdots \phi_{\nu_\varkappa}(x)|^2dx)^{\frac{1}{2}}\lesssim |\nu|^{-\frac{1}{12}}.$ This estimate can be obtained as follows. If $|\nu_i|:=\max_{1\leq j\leq \varkappa}|\nu_j|,$ similar to Remark \ref{infty}  we have
$$\Vert \phi_{\nu_i} \Vert_{L^{\infty}(\mathbb{R}^{n\varkappa})}\lesssim |\nu_i|^{-\frac{1}{12}}\lesssim |\nu|^{-\frac{1}{12}},$$ when $|\nu|$ is large enough. On the other hand, if $k\neq i,$ it follows that
\begin{equation}\label{Diferencia2}
(\int_{\mathbb{R}^n}|\phi_{\nu_1}(x)\cdots \phi_{\nu_\varkappa}(x)|dx)^{\frac{1}{2}}\lesssim \Vert \phi_{\nu_k} \Vert_{L^2}\Vert \phi_{\nu_i} \Vert_{L^\infty}\prod_{j\neq i,k}^{\varkappa}\Vert \phi_{\nu_j} \Vert_{L^{\infty}}\lesssim |\nu|^{-\frac{1}{12}},
\end{equation}
where we have used the crude estimate $\Vert \phi_{\nu_j} \Vert_{L^{\infty}}=O(1)$ for $j\neq k,i$ and that the $L^2-$norm of the function $\phi_{\nu_k}$ is normalised. By using this and Lemma \ref{Lemma10} for $p=1$  we obtain
\begin{align*}&\Vert T_{m(k)}  \Vert_{\mathscr{B}(L^2\times (L^{\infty})^{\varkappa-1},\,L^2)} \\
&\lesssim \sum_{ 2^{k}\leq |\nu|<2^{k+1} } \sup_{x\in\mathbb{R}^n} \int_{\mathbb{R}^{n\varkappa}}|\phi_{\nu}(z)||\mathscr{F}_{H}^{-1}[m_k(x,\cdot)]|(z)dz \prod_{j=2}^{\varkappa} \Vert\phi_{\nu_j}\Vert_{L^{1}} |\nu|^{-\frac{1}{12}}\\
& \lesssim  \sum_{ 2^{k}\leq |\nu|<2^{k+1} } \sup_{x\in\mathbb{R}^n}\Vert \langle z\rangle^{s}\mathscr{F}_{H}^{-1}[m_k(x,\cdot)]\Vert_{L^2(\mathbb{R}^{n\varkappa})} \Vert \langle z\rangle^{-s} \Vert_{L^2(\mathbb{R}^{n\varkappa})} \prod_{j=2}^{\varkappa} |{\nu_j}|^{\frac{n}{4}} |\nu|^{-\frac{1}{12}} \\
& \lesssim  \sum_{ 2^{k}\leq |\nu|<2^{k+1} } 2^{-k(s-\frac{n\varkappa}{2})}\Vert m\Vert_{l.u.\mathcal{H}^s}\vert \nu\vert^{\frac{n}{4}(\varkappa-1)}|\nu|^{-\frac{1}{12}}\asymp 2^{-k(s-\frac{3n\varkappa}{2}-\frac{n}{4}(\varkappa-1) +\frac{1}{12})}.
\end{align*}

Now, we only need to proceed as in the first part, in order to obtain the estimate
\begin{equation}
\Vert T_m \Vert_{\mathscr{B}(L^2\times (L^{\infty})^{\varkappa-1},\,L^2)}\leq C(\Vert m \Vert_{l.u.,\mathcal{H}^s}+\Vert m(\cdot,0) \Vert_{L^{\infty}(\mathbb{R}^{n})} ),
\end{equation}
for $s>\frac{3n\varkappa}{2}+\frac{n}{4}(\varkappa-1)-\frac{1}{12}.$ The last statement can be proved by observing that
\begin{align*}
|(T_{m(k)} f,\overline{g})|
&\leq \sum_{2^k\leq |\nu|<2^{k+1}} \sup_{x\in\mathbb{R}^n} \int_{\mathbb{R}^{n\varkappa}}|\phi_{\nu}(z)||\mathscr{F}_{H}^{-1}[m_k(x,\cdot)]|(z)dz\\
&\times \Vert f_{1}\Vert_{L^{p}}\Vert\phi_{\nu_1}\Vert_{L^{p'}} \prod_{j=2}^{\varkappa} \Vert f_{j}\Vert_{L^{\infty}}\Vert\phi_{\nu_j}\Vert_{L^{1}}\Vert g\Vert_{L^{p'}} (\int_{\mathbb{R}^n}|\phi_{\nu_1}(x)\cdots \phi_{\nu_\varkappa}(x)|^pdx)^{\frac{1}{p}}\\
&\lesssim \sum_{2^k\leq |\nu|<2^{k+1}} \sup_{x\in\mathbb{R}^n} \int_{\mathbb{R}^{n\varkappa}}|\phi_{\nu}(z)||\mathscr{F}_{H}^{-1}[m_k(x,\cdot)]|(z)dz\\
&\times \Vert f_{1}\Vert_{L^{p}}\Vert\phi_{\nu_1}\Vert_{L^{p'}}\Vert\phi_{\nu_1}\Vert_{L^{p}} \prod_{j=2}^{\varkappa} \Vert f_{j}\Vert_{L^{\infty}}\Vert\phi_{\nu_j}\Vert_{L^{1}} \Vert g\Vert_{L^{p'}},
\end{align*}
where in the last line we have used again that the $L^{\infty}$-norm of Hermite functions is $O(1).$ Now, if we denote by $\gamma_p$ the exponent that according to Lemma \ref{Lemma1} satisfies
\begin{equation}\label{Diferencia3}
\Vert \phi_{\nu_1} \Vert_{L^p} \Vert \phi_{\nu_1}\Vert_{L^{p'}}\lesssim \vert \nu_1 \vert^{\gamma_p},
\end{equation}
and we assume that $|\nu_1|:=\max_{1\leq j\leq \varkappa}|\nu_j|$ (which can be obtained by a simple permutation of the $\nu_j's$) we obtain $|\nu_1|\asymp |\nu|,$ and the estimate
$$  \Vert \phi_{\nu_1} \Vert_{L^p} \Vert \phi_\nu\Vert_{L^{p'}} \prod_{j=2}^{\varkappa}\Vert \phi_{\nu_j}\Vert_{L^{1}}\Vert \phi_{\nu_j}\Vert_{L^{\infty}}
\lesssim \vert \nu \vert^{\gamma_p+\frac{(\varkappa-1)(n-1)}{2}} .$$ In the last line according to Lemma \ref{Lemma1} we have used the estimate $\Vert \phi_{\nu_j}\Vert_{L^{1}}\Vert \phi_{\nu_j}\Vert_{L^{\infty}}\lesssim |\nu|^{\frac{n-1}{2}}.$
Now, if we repeat the argument of the first part we obtain
\begin{align*}\Vert T_{m(k)}  \Vert_{\mathscr{B}(L^p\times (L^{\infty})^{\varkappa-1},\,L^p)} &\lesssim  \sum_{ 2^{k}\leq |\nu|<2^{k+1} } 2^{-k(s-\frac{n\varkappa}{2})}\Vert m\Vert_{l.u.\mathcal{H}^s}2^{k\frac{(n-1)(\varkappa-1)}{2}+k\gamma_p}\\
&\asymp 2^{-k(s-\frac{3n\varkappa}{2}-\frac{(n-1)(\varkappa-1)}{2}-\gamma_p)}\Vert m\Vert_{l.u.\mathcal{H}^s} ,
\end{align*}
and consequently the estimate
$$  \Vert T_m \Vert_{\mathscr{B}(L^p\times (L^{\infty})^{\varkappa-1},\,L^p)}\leq C(\Vert m \Vert_{l.u.,\mathcal{H}^s}+\Vert m(\cdot,0) \Vert_{L^{\infty}(\mathbb{R}^{n})} ),$$
for $s>\frac{3n\varkappa}{2}+\frac{(n-1)(\varkappa-1)}{2}+\gamma_p.$
Thus we conclude the proof.
\end{proof}
\begin{remark}
The different regularity orders $s$ imposed to obtain the boundedness of multilinear pseudo-multipliers in  Proposition  \ref{firstlemmamultilinear}   for $p=1,2$ or other values of $p,$ lie in the slight variations that we use for the proof of every specific case. To be more precise, these differences appear as consequence of the conclusions  \eqref{Diferencia1} for $p=1,$ \eqref{Diferencia2} for $p=2,$ (where we have used strongly that the $L^2$-norm of every Hermite functions is normalised) and the estimate \eqref{Diferencia3} when $2<p<\infty.$ Let us mention that our main strategy in the proof of Proposition \ref{multilineartool} will be to use the real interpolation for $p$ between $p_0=1$  and $p_1=2$ or $p$ between $p_1=2$ and arbitrary $p$ with $2<p<\infty,$ together with the different regularity orders imposed in Proposition \ref{firstlemmamultilinear}.
\end{remark}

With a similar proof, as in the previous result, we present the following proposition.

\begin{proposition}\label{secondlemmamultilinear}
Let us consider a pseudo-multiplier $T_m$ defined on $\mathscr{D}(\mathbb{R}^n)^{\varkappa}$ with symbol $m=\{m(x,\nu)\}_{x\in\mathbb{R}^n,\nu\in\mathbb{N}_{0}^{n\varkappa} }$ where $m:\mathbb{R}^n\times \mathbb{R}^{n\varkappa}\rightarrow\mathbb{C}$ satisfies the condition
\begin{equation}\label{multilinearhormander'''}
\Vert m \Vert_{l.u.,H^s}:=\sup_{k>0,\,x\in\mathbb{R}^n} \,2^{k(s-\frac{n\varkappa}{2})}\Vert  \langle z \rangle^{s} \mathscr{F}[m(x,\cdot)\psi(2^{-k}|\cdot|)](z)\Vert_{L^2({\mathbb{R}}^{n\varkappa}_z)}<\infty.
\end{equation}
Then
\begin{itemize}
\item[1.] If $s>\frac{3n\varkappa}{2}+{(\varkappa-1)}\gamma_\infty,$ the operator $ T_m$ extends to a bounded multilinear operator from $L^{1}\times L^{\infty}\times\cdots \times L^{\infty}\times L^{\infty}$ into $L^{1}(\mathbb{R}^n),$ and
\begin{equation}
\Vert T_m \Vert_{\mathscr{B}(L^1\times (L^{\infty})^{\varkappa-1},\,L^1)}\leq C(\Vert m \Vert_{l.u.,\mathcal{H}^s}+\Vert m(\cdot,0) \Vert_{L^{\infty}(\mathbb{R}^n)}).
\end{equation}
\item[2.] If $s>\frac{3n\varkappa}{2}+\frac{(\varkappa-1)n}{4},$ the operator $ T_m$ extends to a bounded multilinear operator from $L^{2}\times L^{\infty}\times\cdots \times L^{\infty}\times L^{\infty}$ into $L^{2}(\mathbb{R}^n),$ and
\begin{equation}
\Vert T_m \Vert_{\mathscr{B}(L^2\times (L^{\infty})^{\varkappa-1},\,L^2)}\leq C(\Vert m \Vert_{l.u.,\mathcal{H}^s}+\Vert m(\cdot,0) \Vert_{L^{\infty}(\mathbb{R}^n)}).
\end{equation}
\item[3.] If  $s>\frac{3n\varkappa}{2}+\frac{(n-1)(\varkappa-1)}{2}+\gamma_p,$ $\gamma_p$  defined as in  \eqref{gammap}, the operator $ T_m$ extends to a bounded multilinear operator from $L^{p}\times L^{\infty}\times\cdots \times L^{\infty}\times L^{\infty}$ into $L^{p}(\mathbb{R}^n),$ and
\begin{equation}
\Vert T_m \Vert_{\mathscr{B}(L^p\times (L^{\infty})^{\varkappa-1},\,L^p)}\leq C(\Vert m \Vert_{l.u.,\mathcal{H}^s}+\Vert m(\cdot,0) \Vert_{L^{\infty}(\mathbb{R}^n)}),
\end{equation}  for all $2< p\leq \infty.$
\end{itemize}
\end{proposition}

Let us note that the second assertion of Proposition \ref{firstlemmamultilinear} requires symbols with regularity order $s>\frac{3n\varkappa}{2}+\frac{(\varkappa-1)n}{4},$ instead of its analogue condition in \eqref{secondlemmamultilinear} where we only need $s>\frac{3n\varkappa}{2}+\frac{(\varkappa-1)n}{4}-\frac{1}{12}.$ This difference is consequence of the different Fourier transforms that we use to  classify the regularity of symbols.

In order to present our multilinear result, we will need the following interpolation theorem which is valid for general measure spaces, but for simplicity we record it on $\mathbb{R}^n.$

\begin{proposition}[Riesz-Thorin Interpolation] Let us assume that a linear operator $T$ can be extended to a bounded operator $T:L^{p_i}(\mathbb{R}^n)\rightarrow L^{q_i}(\mathbb{R}^n)$ for $i\in\{0,1\}.$ If $0<\theta<1$ and $p,q$ are defined by
\begin{equation}
1/p=(1-\theta)/p_0+\theta/p_1,\,\,\,1/q=(1-\theta)/q_0+\theta/q_1,
\end{equation}
the operator $T$ can be extended to a bounded operator $T:L^{p}(\mathbb{R}^n)\rightarrow L^{q}(\mathbb{R}^n)$ with operator norm estimated by
\begin{equation}\label{RTI}
\Vert T\Vert_{\mathscr{B}(L^{p},L^{q})}\leq  \Vert T\Vert_{\mathscr{B}(L^{p_0},L^{q_0})}^{1-\theta} \Vert T\Vert_{\mathscr{B}(L^{p_1},L^{q_1})}^{\theta}.
\end{equation}
\end{proposition}
Although the Riesz-Thorin Interpolation Theorem is a well known result, we will use strongly the control on the norms given in \eqref{RTI}. In the following result we will consider multilinear symbols satisfying H\"ormander conditions of order
$$s>s_{n,\varkappa,p}:=\max\{\frac{3n\varkappa}{2}+\frac{(\varkappa-1)n}{4},\frac{3n\varkappa}{2}+\frac{(n-1)(\varkappa-1)}{2}+\gamma_p\},$$ with $\gamma_p$ defined as in \eqref{gammap}. Let us note that $\frac{3n\varkappa}{2}+\frac{(\varkappa-1)n}{4} $ and $\frac{3n\varkappa}{2}+\frac{(n-1)(\varkappa-1)}{2}+\gamma_p$ can be not compared immediately because the sign of $\gamma_p$ depends on the  values of $p.$

\begin{proposition}\label{multilineartool}
Let $2\leq \varkappa<\infty,$   $\varkappa \in\mathbb{N}_0.$ Let us consider a multilinear pseudo-multiplier $T_m$ defined on $\mathscr{D}(\mathbb{R}^n)^{\varkappa}$ with symbol satisfying \eqref{multilinearhormander} or \eqref{multilinearhormander'''} for  $$s>s_{n,\varkappa,p}:=\max\{\frac{3n\varkappa}{2}+{(\varkappa-1)}\gamma_\infty,\frac{3n\varkappa}{2}+\frac{(\varkappa-1)n}{4}\},\,1\leq p\leq 2,$$ with $\gamma_\infty,$ defined as in \eqref{gammap}. Then the operator
\begin{equation}\label{multi1}
T_m:L^{p_1}\times L^{p_2}\times\cdots \times L^{p_{\varkappa-1}}\times L^{p_\varkappa}\rightarrow L^{p}(\mathbb{R}^n)
\end{equation}
extends to a bounded multilinear operator provided that $1\leq p_{j}\leq \infty,$ $1\leq p\leq 2,$ and $\frac{1}{p}=\frac{1}{p_1}+\cdots +\frac{1}{p_\varkappa}.$ If $m$ satisfies the condition \eqref{multilinearhormander} or \eqref{multilinearhormander'''} for $$s>s_{n,\varkappa,p}:=\max\{\frac{3n\varkappa}{2}+\frac{(\varkappa-1)n}{4},\frac{3n\varkappa}{2}+\frac{(n-1)(\varkappa-1)}{2}+\gamma_p\},$$ with $\gamma_p$ defined as in \eqref{gammap}, then \eqref{multi1} holds true for all $2\leq p\leq \infty$ and $\frac{1}{p}=\frac{1}{p_1}+\cdots +\frac{1}{p_\varkappa}.$
\end{proposition}
\begin{proof}
In order to prove the statement we will use real interpolation together with induction on $\varkappa.$ Let us define
 the set
 \begin{equation}
 M:=\{\varkappa\in\mathbb{N}: \varkappa\geq 2,\textnormal{   and from  \eqref{multilinearhormander} or \eqref{multilinearhormander'''} we deduce \eqref{multi1} for  }s>s_{n,\varkappa,p}\} .
 \end{equation}
 First, we will prove that $\varkappa=2\in M.$ Then, let us assume that a bilinear operator $T_m$ satisfies \eqref{multilinearhormander} or \eqref{multilinearhormander'''}. By Proposition \ref{firstlemmamultilinear} we have that $T_{m}\in \mathscr{B}(L^r\times L^\infty,L^r)$ for $r=1,2,$ provided that
  $$s>s_{n,\varkappa,p}:=\max\{\frac{3n\varkappa}{2}+{(\varkappa-1)}\gamma_\infty,\frac{3n\varkappa}{2}+\frac{(\varkappa-1)n}{4}\}.$$
 Now, if we fix $g_0\in L^{\infty}$ and we consider the operator $$T_{m,0}:=T_m(\cdot,g_0),$$
then  $T_{m,0}\in \mathscr{B}(L^r)$ for $r=1,2,$ and by real interpolation for all $1\leq r\leq 2.$ Moreover, if $r$ is given by
$$  \frac{1}{r}=\frac{1-\theta}{1}+\frac{\theta}{2},$$ for some $0<\theta<1,$ and we taking into account the norm estimates
\begin{equation}
\Vert T_{m,0}\Vert_{\mathscr{B}(L^1)}\leq \Vert T_m\Vert_{\mathscr{B}(L^1\times L^{\infty},L^1)}\Vert g_0 \Vert_{L^\infty},\,\,\Vert T_{m,0}\Vert_{\mathscr{B}(L^2)}\leq \Vert T_m\Vert_{\mathscr{B}(L^2\times L^{\infty},L^2)}\Vert g_0 \Vert_{L^\infty},
\end{equation}
by application of \eqref{RTI} we have
\begin{equation}
\Vert T_{m,0}\Vert_{\mathscr{B}(L^r)}\leq \Vert T_m\Vert^{1-\theta}_{\mathscr{B}(L^1\times L^{\infty},L^1)}\Vert T_m\Vert^{\theta}_{\mathscr{B}(L^2\times L^{\infty},L^2)}\Vert g_0 \Vert_{L^\infty}.
\end{equation}

Consequently we deduce the boundedness of $T_m$ from $L^p\times L^\infty $ into $L^p,$  $1\leq p\leq 2.$ Similarly we obtain the boundedness of $T_m$ from $L^\infty\times L^p $ into $L^p.$ Now if we repeat the argument for every entry of $T_m$, i.e., first fixing the first argument and later fix the second argument, by interpolation we have the boundedness of $T_m$ from $L^{p_1}\times L^{p_2}$ into $L^{p},$  $1\leq p\leq 2,$ with $p_1$ and $p_2$ satisfying the relation
\begin{equation}
\frac{1}{p_1}=\frac{\theta}{p}\,\,\,\,\,\textnormal{   and   }\,\,\,\,\, \frac{1}{p_2}=\frac{1-\theta}{p}
\end{equation}
for some $\theta\in [0,1].$ Clearly, $\frac{1}{p_1}+\frac{1}{p_2}=1.$ Now, we will assume that  every integer number  $s$ less that $\varkappa$ belongs to $M.$ So, let us assume now that we have a multilinear operator $T_{m}$ on $\mathscr{D}(\mathbb{R}^m)^{\varkappa}.$ By Proposition \ref{firstlemmamultilinear} the operator $ T_m$ extends to a bounded multilinear operator from $L^{r}\times L^{\infty}\times\cdots \times L^{\infty}\times L^{\infty}$ into $L^{r}(\mathbb{R}^n),$ for $r=1,2.$ If we consider $g_0\in L^\infty$ and similarly, as in the bilinear case, we define $$T_{m,0}:=T_m(\cdot,\cdot,\cdots,\cdot, g_0),$$
fixing the last argument of $T_{m}$ we obtain a multilinear operator on $\mathscr{D}(\mathbb{R}^m)^{\varkappa-1},$ and by considering that $\varkappa-1\in M,$ we have that
\begin{equation}
T_{m,0}:L^{p_1}\times L^{p_2}\times\cdots \times L^{p_{\varkappa-1}}\rightarrow L^{p}(\mathbb{R}^n)
\end{equation}
extends to a bounded multilinear operator provided that $1\leq p_{j}\leq \infty,$ $1\leq p\leq 2$ and $\frac{1}{p}=\frac{1}{p_1}+\cdots +\frac{1}{p_{\varkappa-1}}.$ As a consequence we obtain that
\begin{equation}
T_{m}:L^{p_1}\times L^{p_2}\times\cdots \times L^{p_{\varkappa-1}}\times L^{\infty}\rightarrow L^{p}(\mathbb{R}^n),
\end{equation}
  is bounded. Because
 \begin{equation}
T_{m}:L^{\infty}\times L^{\infty}\times\cdots \times L^{{\infty}}\times L^{p}\rightarrow L^{p}(\mathbb{R}^n)
\end{equation}
 is bounded for all $1\leq p\leq 2,$
we can fix every argument of $T_m$ and apply the real interpolation in order to provide the boundedness of $T_m$ from $
L^{\dot{p}_1}\times L^{\dot{p}_2}\times\cdots \times L^{\dot{p}_{\varkappa}} $ $L^{p}(\mathbb{R}^n),$
where
$$ \frac{\theta}{{p}_i}=\frac{1}{ \dot{p}_i }\textnormal{     for    }\,1\leq i\leq \varkappa-1,\textnormal{    and    }\frac{1-\theta}{p}=\frac{1}{\dot{p}_\varkappa}, $$
for some $0\leq\theta\leq 1.$ Now, we finish the proof of induction by observing that
\begin{equation}
\sum_{i=1}^{\varkappa}\frac{1}{\dot{p}_i}=\frac{\theta}{p}+\frac{1-\theta}{p}=\frac{1}{p},\textnormal{    and     }\varkappa\in M,
\end{equation}
so, we have proved that $M=\{\varkappa\in\mathbb{N}:\varkappa\geq 2\}.$ Now, in a similar way we can use  statements  2 and 3 of Proposition \ref{firstlemmamultilinear} and the real interpolation (by repetition of the arguments above) in order to provide the boundedness of $T_m$ for the case when $2\leq p<\infty.$
\end{proof}
\begin{remark}
Taking into account that $\gamma_\infty=\frac{n-1}{2}$ for $n\geq 2$ and for $n=1,$  $\gamma_\infty=1/6,$ we can compute explicitly the regularity order  $s_{n,\varkappa,p},$ $1\leq p\leq 2,$ defined in the previous proposition. Indeed, if $n\geq 2,$
$$ s_{n,\varkappa,p}:=\max\{\frac{3n\varkappa}{2}+{(\varkappa-1)}\gamma_\infty,\frac{3n\varkappa}{2}+\frac{(\varkappa-1)n}{4}\}= \frac{3n\varkappa}{2}+\frac{(\varkappa-1)(n-1)}{2},$$
and 
$$ s_{n,\varkappa,p}:=\max\{\frac{3n\varkappa}{2}+{(\varkappa-1)}\gamma_\infty,\frac{3n\varkappa}{2}+\frac{(\varkappa-1)n}{4}\}= \frac{3n\varkappa}{2}+\frac{(\varkappa-1)n}{4}=\frac{3\varkappa}{2}+\frac{\varkappa-1}{4},$$
for $n=1.$ Let us note that these regularity orders cannot be applied to $\varkappa=1,$ in order to recover those regularity orders given in the linear case, because our Proposition \ref{multilineartool} is a consequence of Proposition \ref{firstlemmamultilinear} whose proof uses strongly that $\varkappa\geq 2.$
\end{remark}

\bibliographystyle{amsplain}

\begin{thebibliography}{99}

\bibitem{BagchiThangavelu}  Bagchi, S. Thangavelu, S. On Hermite pseudo-multipliers. J. Funct. Anal. 268 (1) (2015), 140--170.



\bibitem{Bonami-Clerc73} Bonami, A., Clerc, J.L. Somes de Cesaro et multiplicaturs des developments en harmonics spheriques. Trans. Amer. Math. Soc. 183 (1973), 223--263.

\bibitem{Blunk} Blunck, S. A H\"ormander-type spectral multiplier theorem for operators without heat kernel. Ann. Sc. Norm. Super. Pisa Cl. Sci. (5) 2 (2003), no. 3, 449--459.

\bibitem{BuiDuong}  Bui, T. A., Duong, X. T. Besov and Triebel-Lizorkin spaces associated to Hermite operators. J. Fourier Anal. Appl. 21 (2015), no. 2, 405--448.

\bibitem{COSY} Chen, P., Ouhabaz, E. M., Sikora, A., Yan, L. Restriction estimates, sharp spectral multipliers and endpoint estimates for Bochner-Riesz means. J. Anal. Math. 129 (2016), 219--283.

\bibitem{CM1}    Coifman, R.,  Meyer, Y. On commutators of singular integrals and bilinear singular integrals, Trans.
Amer. Math. Soc. 212 (1975), 315--331.

\bibitem{CM2} Coifman, R.,  Meyer, Y. Au del\`a des op\'erateurs pseudo-diff\'erentiels, Ast\'erisque 57 (1978), 1--185.

\bibitem{CM3} Coifman, R.,  Meyer, Y. Commutateurs d’int\'egrales singuli\`eres et op\'erateurs multilin\'eaires, Ann. Inst.
Fourier (Grenoble) 28 (1978), 177--202.



\bibitem{Duo} Duoandikoetxea, J. {Fourier Analysis}, Amer. Math. Soc.  (2001)

\bibitem{Epperson} Epperson, J. Hermite multipliers and pseudo-multipliers, Proc. Amer. Math. Soc. 124
(1996), no. 7, 2061--2068.



\bibitem{Graf3} Grafakos, L., Miyachi, A., Tomita, N. On multilinear Fourier multipliers of limited smoothness. Canad. J. Math. 65 (2013), no. 2, 299--330.

\bibitem{Graf2} Grafakos, L., Nguyen, H. V. Multilinear Fourier multipliers with minimal Sobolev regularity, I. Colloq. Math. 144 (2016), no. 1, 1--30.


\bibitem{Graf} Grafakos, L., Miyachi, A., Nguyen, H. V., Tomita, N. Multilinear Fourier multipliers with minimal Sobolev regularity, II. J. Math. Soc. Japan 69 (2017), no. 2, 529--562.


\bibitem{Hormander1960} H\"ormander, L. (1960) Estimates for translation invariant operators in $L^p$
spaces. Acta Math., 104, 93--140.

\bibitem{Koch} Koch, H., Tataru, D. $L^p$-eigenfunction bounds for the Hermite operator. Duke
Math. J. 128 (2005), 369--392.



\bibitem{Mauceri} Mauceri, G. The Weyl transform and bounded operators on $L^p(\mathbb{R}^n).$ J. Funct. Anal. 39(3), (1980), 408--429.

\bibitem{Mihlin} Mihlin, S.G. On the multipliers of Fourier integrals. Dokl. Akad. Naulc SSSR (N. S.),
109 (1956), 701--703 (Russian).

\bibitem{PeXu} Petrushev, P., Xu, Y. Decomposition of spaces of distributions induced by Hermite expansions. J. Fourier Anal. Appl. 14 (2008), no. 3, 372--414.

\bibitem{Prugovecki} Prugove\u{c}ki, E. Quantum mechanics in Hilbert space. Second edition. Pure and Applied Mathematics, 92. Academic Press, Inc, New York-London, 1981.

\bibitem{ProfRuzM:TokN:20016} Ruzhansky M., Tokmagambetov N., Nonharmonic analysis of boundary value problems, Int. Math. Res. Notices, (2016) 2016 (12), 3548--3615.

\bibitem{ProfRuzM:TokN:20017} Ruzhansky M., Tokmagambetov N., Nonharmonic analysis of boundary value problems without WZ condition, Math. Model. Nat. Phenom., 12 (2017), 115--140.

\bibitem{Ruz} Ruzhansky, M., Turunen, V.,  Pseudo-differential Operators and Symmetries: Background Analysis and Advanced Topics. Birkha\"user-Verlag, Basel, (2010).


\bibitem{Strichartz} Strichartz, R. Multipliers for spherical harmonics expansions. Trans. Amer. Math. Soc. 167 (1972), 115-124.

\bibitem{SYY}
Sikora, A., Yan, L., Yao, X.
Sharp spectral multipliers for operators satisfying generalized Gaussian estimates. J. Funct. Anal. 266 (2014), no. 1, 368--409.

\bibitem{Simon} Simon, B. Distributions and their Hermite expansions. J. Math. Phys. 12 (1971), 140--148.

\bibitem{stempak} Stempak, K. Multipliers for eigenfunction expansions of some Schrödinger operators, Proc. Amer. Math. Soc. 93 (1985), 477--482.

\bibitem{stempak1} Stempak, K., Torreą, J.L. On g-functions for Hermite function expansions, Acta Math. Hung. 109 , 99--125.

\bibitem{stempak2}  Stempak, K., Torreą, J.L. BMO results for operators associated to Hermite expansions, Illinois J. Math. 49 (2005), 1111--1132.

\bibitem{thangavelu0} Thangavelu, S.	
Multipliers for Hermite expansions, Revist. Mat. Ibero. 3 (1987), 1--24.

\bibitem{Thangavelu} Thangavelu, S. Lectures on Hermite and Laguerre Expansions, Math. Notes, vol. 42, Princeton University Press, Princeton, 1993.

\bibitem{Thangavelu2} Thangavelu, S.  Hermite and special Hermite expansions revisited, Duke Math. J., 94(2) (1998), 257--278.

\end{thebibliography}

\end{document}